\documentclass{amsart}

\usepackage{amsfonts, amssymb, amsmath, eucal, verbatim, amsthm, amscd, enumerate}

\usepackage{graphicx}
\usepackage{framed}
\usepackage{tikz}

\newtheorem{theorem}{Theorem}[section]
\newtheorem{lemma}[theorem]{Lemma}
\newtheorem{proposition}[theorem]{Proposition}
\newtheorem{conjecture}[theorem]{Conjecture}

\theoremstyle{definition}

\theoremstyle{remark}
\newtheorem*{remark}{Remark}

\newtheorem*{notation}{Notation}

\newtheorem*{organisation}{Organisation}
\newtheorem*{acknowledgements}{Acknowledgements}

\newcommand{\vertiii}[1]{{\left\vert\kern-0.25ex\left\vert\kern-0.25ex\left\vert #1 
    \right\vert\kern-0.25ex\right\vert\kern-0.25ex\right\vert}}
\newcommand{\R}{{\mathbb R}}

\numberwithin{equation}{section}

\usepackage{setspace}

\begin{document}

\date{\today}

\author[Bez]{Neal Bez}
\address[Neal Bez]{Department of Mathematics, Graduate School of Science and Engineering,
Saitama University, Saitama 338-8570, Japan}
\email{nealbez@mail.saitama-u.ac.jp}
\author[Hong]{Younghun Hong}
\address[Younghun Hong]{Department of Mathematics, Yonsei University, Seoul 03722, Korea}
\email{younghun.hong@yonsei.ac.kr}
\author[Lee]{Sanghyuk Lee}
\address[Sanghyuk Lee]{Department of Mathematical Sciences, Seoul National University, Seoul 151-747, Korea}
\email{shklee@snu.ac.kr}
\author[Nakamura]{Shohei Nakamura}
\address[Shohei Nakamura]{Department of Mathematics and Information Sciences, Tokyo Metropolitan University,
1-1 Minami-Ohsawa, Hachioji, Tokyo, 192-0397, Japan}
\email{nakamura-shouhei@ed.tmu.ac.jp}
\author[Sawano]{Yoshihiro Sawano} 
\address[Yoshihiro Sawano]{Department of Mathematics and Information Sciences, Tokyo Metropolitan University,
1-1 Minami-Ohsawa, Hachioji, Tokyo, 192-0397, Japan}
\email{ysawano@tmu.ac.jp}

\title[Strichartz estimates for orthonormal systems]{On the Strichartz estimates for orthonormal systems of initial data with regularity}

\begin{abstract}
The classical Strichartz estimates for the free Schr\"odinger propagator have recently been substantially generalised to estimates of the form
\[
\bigg\|\sum_j\lambda_j|e^{it\Delta}f_j|^2\bigg\|_{L^p_tL^q_x}\lesssim\|\lambda\|_{\ell^\alpha}
\]
for orthonormal systems $(f_j)_j$ of initial data in $L^2$, firstly in work of Frank--Lewin--Lieb--Seiringer and later by Frank--Sabin. The primary objective is identifying the largest possible $\alpha$ as a function of $p$ and $q$, and in contrast to the classical case, for such estimates the critical case turns out to be $(p,q) = (\frac{d+1}{d},\frac{d+1}{d-1})$. We consider the case of orthonormal systems $(f_j)_j$ in the homogeneous Sobolev spaces $\dot{H}^s$ for $s \in (0,\frac{d}{2})$ and we establish the sharp value of $\alpha$ as a function of $p$, $q$ and $s$, except possibly an endpoint in certain cases, at which we establish some weak-type  estimates. Furthermore, at the critical case $(p,q) = (\frac{d+1}{d-2s},\frac{d(d+1)}{(d-1)(d-2s)})$ for general $s$, we show the veracity of the desired estimates when $\alpha = p$ if we consider frequency localised estimates, and the failure of the (non-localised) estimates when $\alpha = p$; this exhibits the difficulty of upgrading from frequency localised estimates in this context, again in contrast to the classical setting.
\end{abstract}

\maketitle

\section{Introduction and main results}

\subsection{Introduction}The classical Strichartz estimates for the free Schr\"odinger propagator $e^{it\Delta}$ may be stated as\footnote{$A \lesssim B$ means $A \leq CB$ for an appropriate constant $C$}
\begin{equation} \label{e:classical}
\||e^{it\Delta}f|^2\|_{L^p_tL^{q}_x} \lesssim 1
\end{equation}
whenever $\|f\|_{L^2(\mathbb{R}^d)} = 1$, for all spatial dimensions $d \geq 1$ and where $p,q \geq 1$ satisfy $\frac{2}{p} + \frac{d}{q} = d$ and $(p,q,d) \neq (1,\infty,2)$. We note that the endpoint case is $(p,q) = (1,\frac{d}{d-2})$ for $d \geq 3$, proved by Keel and Tao in \cite{KeelTao}, and all other allowable estimates follow by interpolation with the trivial estimate at $(p,q) = (\infty,1)$. When $d=2$, the estimate fails at the endpoint $(p,q) = (1,\infty)$ (see, for example, \cite{MontSmith}), and when $d=1$, the estimate at $(p,q) = (2,\infty)$ is true.

Recently, these estimates have been substantially generalised to the context of orthonormal systems $(f_j)_j$ in $L^2(\mathbb{R}^d)$ in work of Frank--Lewin--Lieb--Seiringer \cite{FLLS} and Frank--Sabin \cite{frank-sabin-1}, resulting in the following.
\begin{theorem}[\cite{FLLS,frank-sabin-1}]\label{t:s=0}
Suppose $d\geq1$. If $p,q\geq1$ satisfy 
$
\frac{2}{p}+\frac{d}{q}=d$,
$
1\leq q<\frac{d+1}{d-1}
$
and $\alpha = \frac{2q}{q+1}$, then
\begin{equation}\label{e:FS}
\bigg\|
\sum_j
\lambda_j
|e^{it\Delta}f_j|^2
\bigg\|_{L^p_tL^q_x}
\lesssim
\|\lambda\|_{\ell^\alpha}
\end{equation}
holds for all orthonormal systems $(f_j)_j$ in $L^2(\mathbb{R}^d)$ and all sequences 
$
\lambda=(\lambda_j)_j$ in $\ell^\alpha(\mathbb{C})
$.
This is sharp in the sense that, for such $p,q$, the estimate fails for all $\alpha > \frac{2q}{q+1}$. Furthermore, when $q = \frac{d+1}{d-1}$, the estimate \eqref{e:FS} holds for all $\alpha < \frac{2q}{q+1}$ and fails when $\alpha = \frac{2q}{q+1}$.
\end{theorem}
To be more precise with regard to attribution, the range $q \in [1,\frac{d+2}{d}]$ was established first in \cite{FLLS}, as well as the necessary condition $\alpha \leq \frac{2q}{q+1}$ and the failure of $(q,\alpha) = (\frac{d+1}{d-1},\frac{2q}{q+1})$. The estimates \eqref{e:FS} in the range $q \in [1,\frac{d+1}{d-1})$ were obtained in \cite{frank-sabin-1}. 

We remark that \eqref{e:FS} may be considered in terms of a square function estimate of the form
\begin{equation*}
\bigg\|
\bigg(\sum_j |e^{it\Delta}f_j|^2 \bigg)^{1/2} \bigg\|_{L^{2p}_tL^{2q}_x}
\lesssim
\bigg(\sum_j \|f_j\|_2^{2\alpha} \bigg)^{1/2\alpha}
\end{equation*}
for \emph{orthogonal} systems $(f_j)_j$ in $L^2(\mathbb{R}^d)$. For a number of reasons, the formulation in \eqref{e:FS} is more convenient; for example, later we make use of a semi-classical limiting argument to connect such estimates to Strichartz estimates for the velocity average $\rho F$ of the solution $F$ of the kinetic transport equation (see the forthcoming Proposition \ref{p:semiclassical}) and from this viewpoint, \eqref{e:FS} is more natural.

The idea of extending classical functional inequalities to orthonormal systems of input functions goes back to famous work of Lieb--Thirring \cite{Lieb-Thirring-1}, where a generalisation of a certain Gagliardo--Nirenberg--Sobolev estimate to orthonormal systems of $L^2$ functions was established. The Lieb--Thirring inequalities are a key component in the proof of stability of matter; see, for example, \cite{Lieb-Thirring-1} or the comprehensive survey by Lieb \cite{LiebBAMS} for further details. In \cite{Lieb_Sobolev}, Lieb also obtained the estimate
\begin{equation} \label{e:Lieb_Sobolev}
\bigg\|\sum_j \lambda_j ||D|^{-s}f_j |^2 \bigg\|_{L^{q}(\mathbb{R}^d)} \lesssim \|\lambda\|_{\ell^1}^{\frac{1}{q}} \|\lambda\|_{\ell^\infty}^{\frac{1}{q'}}
\end{equation}
for orthonormal systems $(f_j)_j$ in $L^2(\mathbb{R}^d)$, where $q \in (1,\infty)$ and $2s = d - \frac{d}{q}$. Here, and throughout this paper, we use the notation $|D| = \sqrt{-\Delta}$. For a single function input, \eqref{e:Lieb_Sobolev} reduces to a classical Sobolev embedding estimate. The driving motivation for extending fundamental estimates to orthonormal systems has come from quantum mechanics, since such systems give a description of independent fermions in euclidean space. We refer the reader to \cite{FLLS} and \cite{frank-sabin-1} for further details, along with work of Lewin--Sabin in \cite{LewinSabinWP} and \cite{LewinSabinScatt}, where the estimates in Theorem \ref{t:s=0} were applied to the theory of the Hartree equation for an infinite number of particles (see also the survey by Sabin \cite{sabin}, along with \cite{AJM}, \cite{CHP-1} and \cite{CHP-2} for related results).

Regarding the exponent $\alpha$ in Theorem \ref{t:s=0}, note that the triangle inequality and classical Strichartz estimate \eqref{e:classical} imply  
\begin{equation} \label{e:alpha=1}
\bigg\|
\sum_j
\lambda_j
|e^{it\Delta}f_j|^2
\bigg\|_{L^p_tL^q_x}
\leq
\sum_{j}
|\lambda_j|
\left\|
|e^{it\Delta}f_j|^2
\right\|_{L^p_tL^q_x}
\lesssim
\sum_j|\lambda_j|
\end{equation}
which gives \eqref{e:FS} with $\alpha=1$ \emph{without} making use of the orthogonality; the pertinent point here is to raise $\alpha$ as far as possible by capitalising on the orthogonality of the $f_j$. Of course, the above ``trivial" argument in \eqref{e:alpha=1} can be used for a larger range of $q$ than the range $q \in [1,\frac{d+1}{d-1})$ in Theorem \ref{t:s=0}, but interestingly, at the Keel--Tao endpoint $(p,q) = (1,\frac{d}{d-2})$ where $q$ is as large as possible (in this discussion, we are assuming $d \geq 3$), the exponent $\alpha = 1$ cannot be improved (see \cite{frank-sabin-2}). It follows that $q = \frac{d+1}{d-1}$ plays the role of an endpoint in the context of \eqref{e:FS}. Indeed,  interpolating \eqref{e:FS} for $q$ arbitrarily close to $\frac{d+1}{d-1}$ (and $\alpha = \frac{2q}{q+1}$) with $q = \frac{d}{d-2}$ (and $\alpha = 1$) gives \eqref{e:FS} for all $q \in (\frac{d+1}{d-1},\frac{d}{d-2})$ and any $\alpha < p$; on the other hand, it was shown in \cite{frank-sabin-2} that \eqref{e:FS} fails for $\alpha > p$. 
\begin{theorem}[\cite{frank-sabin-2}] \label{t:FS2}
Suppose $d\geq 3$ and $p,q\geq1$ satisfy 
$
\frac{2}{p}+\frac{d}{q}=d$,
and
$
\frac{d+1}{d-1} < q < \frac{d}{d-2}.
$
Then, for any $\alpha < p$,
\begin{equation}\label{e:FSeps}
\bigg\|
\sum_j
\lambda_j
|e^{it\Delta}f_j|^2
\bigg\|_{L^p_tL^q_x}
\lesssim \|\lambda\|_{\ell^{\alpha}}
\end{equation}
holds for all orthonormal systems $(f_j)_j$ in $L^2(\mathbb{R}^d)$ and all sequences 
$
\lambda=(\lambda_j)_j$ in $\ell^{\alpha}(\mathbb{C})
$.
This is sharp in the sense that the estimate fails for all $\alpha > p$.
\end{theorem}
Thus, for $d \geq 3$ and $q \in (\frac{d+1}{d-1},\frac{d}{d-2})$, the only remaining issue is the critical case $\alpha = p$; such estimates would follow by interpolation if the following interesting conjecture (raised in \cite{FLLS}; see also \cite{frank-sabin-2}) were true. 
\begin{conjecture}\label{conj:FLLS}
Let $d\geq1$. At the endpoint $(p,q) = (\frac{d+1}{d},\frac{d+1}{d-1})$, the restricted-type estimate
\begin{equation} \label{e:FLLSconj}
\bigg\|
\sum_j
\lambda_j
|e^{it\Delta}f_j|^2
\bigg\|_{L^p_tL^q_x}
\lesssim
\|\lambda\|_{\ell^{p,1}}
\end{equation}
holds for all orthonormal systems $(f_j)_j$ in $L^2(\mathbb{R}^d)$ and all sequences $\lambda=(\lambda_j)_j$ in $\ell^{p,1}(\mathbb{C})$.
\end{conjecture}
Here, $\ell^{p,1}(\mathbb{C})$ is a Lorentz sequence space, and we clarify the meaning of this in the next section. We remark that the only argument we are aware of to obtain \eqref{e:FSeps} with $\frac{2}{p}+\frac{d}{q}=d$, $\frac{d+1}{d-1} < q < \frac{d}{d-2}$ and $\alpha = p$ from the estimate \eqref{e:FLLSconj} (were it to be true) proceeds via real interpolation with the estimate \eqref{e:FSeps} when $(p,q,\alpha) = (1,\frac{d-2}{d},1)$. Such an argument is not completely obvious since real interpolation, in general, does not work well with mixed-norm Lebesgue spaces (see, for example, \cite{Cwikel74}); in certain cases, real interpolation of mixed-norm spaces gives the expected outcome and we shall in fact use such cases in the proof of one of our main results in Theorem \ref{t:smooth} below.

\subsection{Orthonormal data in Sobolev spaces} 
Our goal in this paper is to provide a more complete picture of the generalised Strichartz estimates for orthonormal systems. Firstly, we establish sharp estimates of the form
\begin{equation} \label{e:warmupgoal}
\bigg\|
\sum_j
\lambda_j
|e^{it\Delta}|D|^{-s}f_j|^2
\bigg\|_{L^p_tL^q_x}
\lesssim
\|\lambda\|_{\ell^{\alpha}}
\end{equation}
for orthonormal systems in $L^2(\mathbb{R}^d)$, or equivalently, estimates of the form \eqref{e:FS} for orthonormal systems in the (homogeneous) Sobolev space $\dot{H}^s(\mathbb{R}^d)$; for such an estimate to be true, we need to assume the scaling condition
\begin{equation} \label{e:scaling}
\frac{2}{p} + \frac{d}{q} = d - 2s. 
\end{equation}
Of course, the single-function classical counterpart to such an estimate is
\begin{equation} \label{e:smoothStr}
\||e^{it\Delta}f|^2\|_{L^p_tL^q_x} \lesssim 1 \qquad \text{whenever $\|f\|_{\dot{H}^s(\mathbb{R}^d)} = 1$}
\end{equation}
for which it is well known that \eqref{e:scaling} is necessary (by a scaling argument) as well as the condition $s \in [0,\frac{d}{2})$. 

A precise understanding of the interaction between the smoothness parameter $s$ and the exponent $\alpha$ will be derived, thus providing a natural extension of Theorems \ref{t:s=0} and \ref{t:FS2} to all admissible $s$. This question naturally arises if we consider Lieb's generalisation of the classical Sobolev estimate in \eqref{e:Lieb_Sobolev}; indeed, this implies
\begin{equation*}
\bigg\|\sum_j \lambda_j |e^{it\Delta}|D|^{-s}f_j |^2 \bigg\|_{L^\infty_tL^{q}_x} \lesssim \|\lambda\|_{\ell^1}^{\frac{1}{q}} \|\lambda\|_{\ell^\infty}^{\frac{1}{q'}}
\end{equation*}
for all orthonormal systems $(f_j)_j$ in $L^2(\mathbb{R}^d)$ (since $(e^{it\Delta}f_j)_j$ is also an orthonormal system for each fixed $t \in \mathbb{R}$). Here, $s \in [0,\frac{d}{2})$ and $2s = d - \frac{d}{q}$. However, it is not clear to us how to induce estimates of the form \eqref{e:warmupgoal} from this for general $p$ and $q$ with the sharp $\alpha$. If Lieb's estimate \eqref{e:Lieb_Sobolev} were achievable with the smaller quantity $\|\lambda\|_{\ell^q}$ on the right-hand side, then we would be able to obtain our desired goal; however, such an estimate fails (we will observe a somewhat stronger negative result in Proposition \ref{p:OBr>1fail} below.)

Prior to stating our main results, we offer some words on why the estimates \eqref{e:warmupgoal} for $s > 0$ offer some (perhaps unexpected) difficulties. In the case of the single-function estimate \eqref{e:smoothStr}, one can proceed by first establishing the desired estimates for initial data which are frequency localised to annuli and upgrade to general data via Littlewood--Paley theory. It seems difficult to proceed in this way in the case of the generalised estimates \eqref{e:warmupgoal} and we shall highlight this by showing, somewhat roughly speaking, that frequency localised estimates are true in almost all cases at the critical value of $\alpha$, whereas on a certain critical line, we shall show that the desired estimates without the frequency localisation are not true. Despite this, we are able to obtain  estimates of the form \eqref{e:warmupgoal} with the sharp value of $\alpha$ (expect endpoints in certain cases) and our argument is based on upgrading the frequency localised estimates to general data and is carried out via a succession of interpolation arguments.

\begin{figure} \label{fig:OtoG}
\includegraphics[width=300pt]{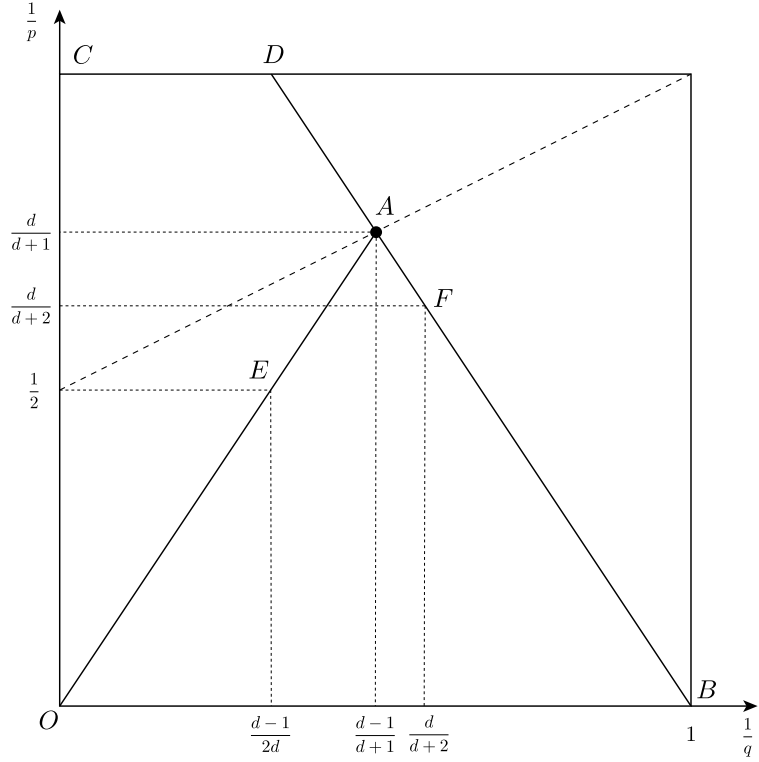}
\vspace{-5mm}
\caption{{ 
The points $A$ to $F$}}
\end{figure}

In order to state our results precisely, we establish some notation.
\begin{notation}
We introduce the following points (see Figure 1
): 
\begin{equation*}
\begin{array}{llllll}
& A = (\tfrac{d-1}{d+1},\tfrac{d}{d+1}) \quad & B = (1,0) \quad & C = (0,1) \\
& D = (\tfrac{d-2}{d},1) & E = (\tfrac{d-1}{2d},\tfrac{1}{2}) \quad & F  = (\tfrac{d}{d+2},\tfrac{d}{d+2})
\end{array}
\end{equation*}
and the origin $O = (0,0)$. For points $X_j \in \mathbb{R}^2$, $j=1,2,3,4$, we write
\begin{align*}
[X_1,X_2] & = \{ (1-t)X_1 + tX_2 : t \in [0,1]\} \\
[X_1,X_2) & = \{ (1-t)X_1 + tX_2 : t \in [0,1)\} \\
(X_1,X_2) & = \{ (1-t)X_1 + tX_2 : t \in (0,1)\}
\end{align*}
for line segments connecting $X_1$ and $X_2$, including or excluding $X_1$ and $X_2$ as appropriate. We write
$
X_1X_2X_3 
$
for the convex hull of $X_1, X_2, X_3$, and
$
{\rm int}\, X_1X_2X_3 
$
for the interior of $X_1X_2X_3$. Similarly, 
$
X_1X_2X_3X_4
$
denotes the convex hull of $X_1, X_2, X_3,X_4$, and
$
{\rm int}\, X_1X_2X_3X_4 
$
denotes the interior of $X_1X_2X_3X_4$. In particular,
\[
{\rm int}\, OAB = \bigg\{ \bigg(\frac{1}{q},\frac{1}{p}\bigg) \in (0,1)^2 : \text{$\frac{1}{q} > \frac{d}{(d-1)p}$ and $\frac{2}{p} + \frac{d}{q} < d$}\bigg\}
\]
and
\[
{\rm int}\, OCDA = \bigg\{ \bigg(\frac{1}{q},\frac{1}{p}\bigg) \in (0,1)^2 : \text{$\frac{1}{q} < \frac{d}{(d-1)p}$ and $\frac{2}{p} + \frac{d}{q} < d$}\bigg\}.
\]

We remark that the line segment $[B,A)$ corresponds to the range of estimates in Theorem \ref{t:s=0}, and (for $d \geq 3$) the segment $(A,D)$ corresponds to the range of estimates in Theorem \ref{t:FS2}. 

Throughout the paper, we need the exponent $\alpha^*(p,q)$ determined by 
\[
\frac{d}{\alpha^*(p,q)}=\frac{1}{p}+\frac{d}{q}.
\]
Note that if $\frac{2}{p} + \frac{d}{q} = d$ (corresponding to the case $s=0$) then $\alpha^*(p,q) = \frac{2q}{q+1}$, which is the sharp exponent in Theorem \ref{t:s=0}. Also, note that if $(\frac{1}{q},\frac{1}{p})$ belongs to the line segment $[O,A]$, then $\alpha^*(p,q) = p$, and if $(\frac{1}{q},\frac{1}{p})$ belongs to the line segment $[O,B]$, then $\alpha^*(p,q) = q$.
\end{notation}
\subsection{Frequency localised estimates}
The following theorem contains our frequency localised estimates which are shown to be true in almost all admissible cases with the sharp value of $\alpha$. In this statement, $P$ is the operator given by $\widehat{Pf}(\xi)=\phi(\xi)\widehat{f}(\xi)$, where $\phi$ is any nontrivial function belonging to $C^\infty_c(\mathbb{R}^d)$. 
\begin{theorem} \label{t:localmain} 
Suppose $d \geq 3$.
\begin{enumerate} \item[(1)] If $(\frac{1}{q},\frac{1}{p})$ belongs to $OAB \setminus A$, and $\alpha = \alpha^*(p,q)$, then 
\begin{equation} \label{e:localmain} 
\bigg\|\sum_j\lambda_j |e^{it\Delta}Pf_j|^2\bigg\|_{L^p_tL^q_x} \lesssim_\phi \|\lambda\|_{\ell^\alpha}
\end{equation}
holds for all orthonormal systems $(f_j)_j$ in $L^2(\mathbb{R}^d)$ and all sequences $\lambda=(\lambda_j)_j$ in $\ell^\alpha(\mathbb{C})$. This is sharp in the sense that the estimate fails for all $\alpha > \alpha^*(p,q)$. 
\item[(2)] If $(\frac{1}{q},\frac{1}{p})$ belongs to ${\rm int}\,OCDA$ and $\alpha = p$, then 
\begin{equation*}
\bigg\|\sum_j\lambda_j |e^{it\Delta}Pf_j|^2\bigg\|_{L^p_tL^q_x} \lesssim_\phi \|\lambda\|_{\ell^\alpha}
\end{equation*}
holds for all orthonormal systems $(f_j)_j$ in $L^2(\mathbb{R}^d)$ and all sequences $\lambda=(\lambda_j)_j$ in $\ell^\alpha(\mathbb{C})$. This is sharp in the sense that the estimate fails for all $\alpha > p$. 
\end{enumerate}
\end{theorem} 
The above theorem provides sharp estimates on all of the admissible region except the line segment $[A,D]$, where the estimate \eqref{e:localmain} is equivalent to the corresponding (strong-type) estimate without the localisation operator $P$ via a simple scaling argument; in this case, the strong-type estimate at the critical $\alpha = p$ remains open (see Theorems \ref{t:smooth} and \ref{t:smoothLorentz} below).


In order to establish the necessary condition $\alpha \leq \min\{\alpha^*(p,q),p\}$ in Theorem \ref{t:localmain}, we construct two explicit orthonormal systems of initial data $(f_j)_j$ and, using simple arguments, derive the claimed necessary condition. In each case,  the initial data have frequency support in some fixed annulus (not necessarily centered at the origin)  and thus will be used to derive the same necessary condition in Theorem \ref{t:smooth} below. For $s=0$, this recovers the necessary condition $\alpha\leq \min{(\frac{2q}{q+1},p)}$ contained in Theorems \ref{t:s=0} and \ref{t:FS2}; however, in both \cite{FLLS} and \cite{frank-sabin-2}, the proofs were operator-theoretic and somewhat more involved.

On the line $[B,A)$, the estimates in Theorem \ref{t:localmain} follow from the work of Frank--Sabin in \cite{frank-sabin-1} (since orthonormality is not preserved by the action of $P$, we cannot directly apply Theorem \ref{t:s=0}; however, the argument in \cite{frank-sabin-1} is still applicable). On $[C,D]$ the estimates trivially hold with $\alpha = 1$, which means that it will suffice to prove \eqref{e:localmain} on the critical line $[O,A)$.

The estimates in Theorem \ref{t:localmain} along the critical line $[O,A)$ are delicate and we establish these using bilinear real interpolation in the spirit of the proof of the endpoint case for the classical estimate \eqref{e:classical} in \cite{KeelTao}. Furthermore, on the critical line $[O,A]$, we shall also show below that the corresponding estimate to \eqref{e:warmupgoal} without the frequency localisation operator $P$ \emph{fails} to hold, which shows that certain difficulties arise when attempting to globalise the estimates in Theorem \ref{t:localmain}. We shall, in fact, show that in the scale of Lorentz spaces, only the very weakest estimate (restricted weak-type) is possible on $[O,A]$ (see the forthcoming Proposition \ref{p:OAfail}).
\subsection{Strong-type estimates}
Despite the difficulties raised above, using the frequency localised estimates from Theorem \ref{t:localmain}, we are able to prove the following strong-type estimates.
\begin{theorem}\label{t:smooth}
\begin{enumerate} \item[(1)] Let $d \geq 1$ and suppose $(\frac{1}{q},\frac{1}{p})$ belongs to ${\rm int}\,OAB$. If $2s = d - (\frac{2}{p} + \frac{d}{q})$ and $\alpha = \alpha^*(p,q)$, then 
\begin{equation*} \label{e:smoothONS}
\bigg\|\sum_j\lambda_j |e^{it\Delta}f_j|^2\bigg\|_{L^p_tL^q_x} \lesssim \|\lambda\|_{\ell^\alpha}
\end{equation*}
holds for all orthonormal systems $(f_j)_j$ in $\dot{H}^s(\mathbb{R}^d)$ and all sequences $\lambda=(\lambda_j)_j$ in $\ell^\alpha(\mathbb{C})$. This is sharp in the sense that the estimate fails for all $\alpha > \alpha^*(p,q)$. 
\item[(2)]  Let $d \geq 2$ and suppose $(\frac{1}{q},\frac{1}{p})$ belongs to ${\rm int}\,OCDA$. If $2s = d - (\frac{2}{p} + \frac{d}{q})$ and $\alpha < p$, then 
\begin{equation*}
\bigg\|\sum_j\lambda_j |e^{it\Delta}f_j|^2\bigg\|_{L^p_tL^q_x} \lesssim \|\lambda\|_{\ell^\alpha}
\end{equation*}
holds for all orthonormal systems $(f_j)_j$ in $\dot{H}^s(\mathbb{R}^d)$ and all sequences $\lambda=(\lambda_j)_j$ in $\ell^\alpha(\mathbb{C})$. This is sharp in the sense that the estimate fails for all $\alpha > p$. 
\end{enumerate}
\end{theorem}
When $d \geq 3$, the admissible exponents for \eqref{e:smoothStr} correspond to the region $OCDB \setminus [O,C]$ and when $d=2$ the region is $OCB \setminus [O,C]$. For $d=1$, the admissible exponents correspond to the region $OAB \setminus [O,A)$; thus, in the case of one spatial dimension, the considerations in Theorem \ref{t:smooth}(2) do not arise. We also note that the sufficiency claim in Theorem \ref{t:smooth}(2) follows quickly from the sufficiency claim in Theorem \ref{t:smooth}(1) by complex interpolation, so its statement is included above as a matter of completeness.

Already in the above discussion, we have indicated why the strong-type estimates in Theorem \ref{t:smooth} do not seem to follow easily from known results. Another attempt to succeed in such a manner would be to naively mimick the deduction of the single-function estimates \eqref{e:smoothStr} from \eqref{e:classical} via the classical Sobolev embedding theorem; using, instead, the vector-valued generalisation of the Sobolev embedding theorem, followed by Theorems \ref{t:s=0} and \ref{t:FS2}, we may obtain
\[
\bigg\|\sum_j\lambda_j |e^{it\Delta}|D|^{-s}f_j|^2\bigg\|_{L^p_tL^q_x}
\lesssim
\bigg\|\sum_j\lambda_j|e^{it\Delta}f_j|^2\bigg\|_{L^p_tL^{\widetilde{q}}_x}
\lesssim
\|\lambda\|_{\ell^{\alpha}},
\]
where $\frac{d}{\widetilde{q}} = 2s + \frac{d}{q}$, $\alpha=\frac{2\widetilde{q}}{\widetilde{q}+1}$ for $\tilde{q} \in [1,\frac{d+1}{d-1})$ and $\alpha = p-\varepsilon$ for $\tilde{q} \in [\frac{d-1}{d+1},1]$ for any $\varepsilon > 0$ (more precisely, we can take $\varepsilon = 0$ when $\widetilde{q} = 1$). This value of $\alpha$ is sharp (modulo $\varepsilon$) only for $p \in [1,\frac{d+1}{d}]$, but very far from sharp as $p$ increases beyond $\frac{d+1}{d}$. For example, as the regularity parameter $s$ approaches $\frac{d}{2}$, the sharp value $\alpha = \alpha^*(p,q)$ we obtain in Theorem \ref{t:smooth} approaches infinity, whereas the above argument yields $\alpha =\frac{2\widetilde{q}}{\widetilde{q}+1}$ which is bounded above by 2. 

Finally, we remark that for $(\frac{1}{q},\frac{1}{p})$ on the line segment between $(0,\frac{1}{2})$ and $E$ (excluding the endpoints), the estimates in Theorem \ref{t:smooth} can be obtained from \cite{CHP-2}. In this case $\alpha = 2$ and this allows substantial simplification (via the duality principle in Proposition \ref{p:dualityprinciple} below, this case corresponds to an estimate in the Schatten space $\mathcal{C}^2$).

\subsection{Weak-type estimates}
In light of the above results, the remaining important problem is whether the estimates \eqref{e:warmupgoal} are valid when $(\frac{1}{q},\frac{1}{p})$ belongs to $OCDA$ with $\alpha = p$. Using the frequency localised estimates in Theorem \ref{t:localmain} and a certain interpolation argument (see Proposition \ref{p:Bourgaintrick} below) we can deduce that restricted weak type estimates
\begin{equation} \label{e:RWTmain}
\bigg\|\sum_j\lambda_j|e^{it\Delta}f_j|^2\bigg\|_{L^{p,\infty}_tL^q_x}
\lesssim
\|\lambda\|_{\ell^{\alpha,1}}
\end{equation}
are valid throughout ${\rm int}\, OCDB$, with the sharp value of $\alpha$ (in particular, $\alpha = p$ in the region ${\rm int}\,OCDA$) for orthonormal systems $(f_j)_j$ in $\dot{H}^s(\mathbb{R}^d)$. Since we are in the context of mixed-norm estimates, as we have already mentioned, the method of real interpolation does not work well in general (see, for example, \cite{Cwikel74}) and it seems difficult to upgrade \eqref{e:RWTmain} to strong-type estimates without losing optimality of the exponent $\alpha$.

With considerable more effort, we are able to establish the following result containing certain weak-type estimates at the critical exponent $\alpha = p$ on $(D,A)$.
\begin{theorem} \label{t:smoothLorentz}
Let $d \geq 2$. If $(\frac{1}{q},\frac{1}{p})$ belongs to $(D,A)$, then
\begin{equation*}
\bigg\|\sum_j\lambda_j |e^{it\Delta}f_j|^2\bigg\|_{L^{p,\infty}_tL^q_x} \lesssim \|\lambda\|_{\ell^{p,\infty}}
\end{equation*}
holds for all orthonormal systems $(f_j)_j$ in $L^2(\mathbb{R}^d)$ and all sequences $\lambda=(\lambda_j)_j$ in $\ell^{p,\infty}(\mathbb{C})$.
\end{theorem}
Here, $L^{p,\infty}$ and $\ell^{p,\infty}$ denote weak $L^p$ and weak $\ell^p$, respectively; in the next section, we clarify the meaning of this notation. Our argument for proving Theorem \ref{t:smoothLorentz} also proceeds via bilinear real interpolation in the spirit of \cite{KeelTao}.

It seems reasonable to believe that the estimates in Theorem \ref{t:FS2} could be true at the critical exponent $\alpha = p$ and away from the endpoint $(p,q) = (\frac{d+1}{d},\frac{d+1}{d-1})$; with this in mind, Conjecture \ref{conj:FLLS} becomes particularly tantalising. Although we are unable to resolve the conjecture in general dimensions $d \geq 2$, we give a short proof that it fails when $d=1$. Our argument exploits the connection between estimates of the form \eqref{e:FS} and the Strichartz estimates for the solution of the kinetic transport equation via a semi-classical limiting argument. The corresponding conjecture for the kinetic transport equation was raised in \cite{BBGL_CPDE} and we actually show the stronger result that this also fails when $d=1$ (see the forthcoming Theorem \ref{t:KTweakfail}). 

\begin{organisation}
Before entering the proofs of the above results, we begin in the next section with various preliminaries and establish some notation. The frequency localised estimates in Theorem \ref{t:localmain} will be proved in Section \ref{section:local} and the strong-type estimates in Theorem \ref{t:smooth} will be proved in Section \ref{section:strong}. In Section \ref{section:OA} we establish some negative results on the critical line $[O,A]$. This includes the fact that on $[O,A)$ the strong-type estimates fail at the critical value of $\alpha$, and we also show that Conjecture \ref{conj:FLLS} fails when $d=1$. The weak-type estimates in Theorem \ref{t:smoothLorentz} will be proved in Section \ref{section:AD}. Finally, in Section \ref{section:further}, we record some further results, including some considerations on the line segment $[O,B]$ relating to Lieb's generalised Sobolev estimate \eqref{e:Lieb_Sobolev}. We also include the observation that analogous results for the Schr\"odinger equation for the harmonic oscillator, corresponding to $e^{-itH}$, where $H = -\Delta + |x|^2$, are valid and may be obtained directly from the results for $e^{it\Delta}$ via a simple transformation.
\end{organisation}

\section{Preliminaries} \label{section:pre}

For appropriate functions $f : \mathbb{R}^d \to \mathbb{C}$, we denote the Fourier transform of $f$ by
\[
\widehat{f}(\xi) = \int_{\mathbb{R}^d} f(x) e^{-ix \cdot \xi} \, \mathrm{d}x
\]
and $\dot{H}^s(\mathbb{R}^d)$ is the homogeneous Sobolev space with norm 
\[
\|f\|_{\dot{H}^s(\mathbb{R}^d)} = \||D|^s f \|_{L^2(\mathbb{R}^d)} = \frac{1}{(2\pi)^{\frac{d}{2}}} \bigg(\int_{\mathbb{R}^d} |\widehat{f}(\xi)|^2 |\xi|^{2s} \, \mathrm{d}\xi \bigg)^{1/2}.
\]

\subsection{Lorentz spaces}
Here, we give a short introduction to the Lorentz spaces $L^{p,r}$ and the Lorentz sequence spaces $\ell^{p,r}$; for further details, we refer the reader to \cite{SteinWeiss}.

Considering $\mathbb{R}^n$ with Lebesgue measure $|\cdot|$, then we write $L^{p,r} = L^{p,r}(\mathbb{R}^n)$ for the Lorentz space of measurable functions $f$ on $\mathbb{R}^n$ with $\|f\|_{L^{p,r}} < \infty$, where
\[
\|f\|_{L^{p,r}} = \bigg(\int_0^\infty (t^{1/p} f^*(t))^r \, \frac{\mathrm{d}t}{t} \bigg)^{1/r}
\]
for $p,r \in [1,\infty)$, and $\|f\|_{L^{p,\infty}} = \sup_{t > 0} t^{1/p}f^*(t)$ for $p \in [1,\infty]$. Here, $f^*$ is the decreasing rearrangement of $f$ defined by
\[
f^*(t) = \inf\{\mu \geq 0 : a_f(\mu) \leq t\}
\]
where $a_f$ is the distribution function of $f$ given by
\[
a_f(\mu) = | \{ x \in \mathbb{R}^n : |f(x)| > \mu \}.
\]
Equivalently,
\[
\|f\|_{L^{p,r}} = p^{\frac{1}{r}} \bigg(\int_0^\infty \, (\mu a_f(\mu)^{\frac{1}{p}})^r \frac{\mathrm{d}\mu}{\mu} \bigg)^{1/r}
\]
for $p,r \in [1,\infty)$, and $\|f\|_{L^{p,\infty}} = \sup_{\mu > 0} \mu a_f(\mu)^{\frac{1}{p}}$.

For $p \in (1,\infty)$ and $r \in [1,\infty]$, the space $L^{p,r}$ is normable. The function $\|\cdot\|_{L^{p,r}}$ defined above gives rise to a norm when $r \leq p$ and a quasi-norm otherwise. To obtain a norm in all cases, equivalent to $\|\cdot\|_{L^{p,r}}$, we define
\[ 
\|f\|_{L^{p,r}}^*= \bigg(\int_0^\infty (t^{1/p} f^{**}(t))^r \, \frac{\mathrm{d}t}{t} \bigg)^{1/r}
\] 
for $p,r \in [1,\infty)$, and $\|f\|_{L^{p,\infty}} = \sup_{t > 0} t^{1/p}f^{**}(t)$ for $p \in [1,\infty]$, where
\[
f^{**} (t)=\frac1t\int^t_0 f^*(s)\, \mathrm{d}s.
\]   
One can then prove that, if $p \in (1,\infty)$ and $r \in [1,\infty]$, then $\|\cdot\|_{L^{p,r}}^*$ is a norm and satisfies
\begin{equation} \label{e:normequiv}
\|f\|_{L^{p,r}} \leq \|f\|_{L^{p,r}}^* \leq \frac{p}{p-1}\|f\|_{L^{p,r}}
\end{equation}
for all $f \in L^{p,r}$ (see, for example, \cite{SteinWeiss}).

Finally, we introduce the Lorentz sequence space $\ell^{\alpha,r}$ as the space of all sequences $\lambda = (\lambda_j)_j \in c_0$ such that $\|\lambda\|_{\ell^{\alpha,r}} < \infty$, where
\[
\|\lambda\|_{\ell^{\alpha,r}} = \bigg( \sum_{j=1}^\infty (j^{1/\alpha} \lambda^*_j)^r \frac{1}{j} \bigg)^{1/r}
\]
for $\alpha,r \in [1,\infty)$, and $\|\lambda\|_{\ell^{\alpha,\infty}} = \sup_{j \geq 1} j^{1/\alpha}\lambda^*_j$ for $\alpha \in [1,\infty]$. Here, $\lambda^* = (\lambda^*_j)_j$ is the sequence $(|\lambda_j|)_j$ permuted in a decreasing order.

\subsection{Schatten spaces and a duality principle}
We shall make use of the following duality principle several times, which recasts the estimates appearing in the various statements in the previous section in terms of Schatten space bounds on operators of the form $We^{it\Delta}(e^{it\Delta})^*\overline{W}$. This can be found in \cite{frank-sabin-1} in the case of Lebesgue spaces, and here we note that it extends to Lorentz spaces with trivial modifications to the proof.
\begin{proposition} \label{p:dualityprinciple}
Suppose $p,q \geq 1$ and $r, \widetilde{r}, \alpha, \beta \geq 1$. Also, let $Uf(x,t) = e^{it\Delta}f(x)$. Then 
\begin{equation*}
\bigg\|\sum_j\lambda_j |Uf_j|^2\bigg\|_{L^{p,r}_tL^{q,\widetilde{r}}_x} \lesssim \|\lambda\|_{\ell^{\alpha,\beta}}
\end{equation*}
holds for all orthonormal systems $(f_j)_j$ in $L^2(\mathbb{R}^d)$ and all sequences $\lambda=(\lambda_j)_j$ in $\ell^{\alpha,\beta}(\mathbb{C})$, if and only if
\[
\|WUU^*\overline{W}\|_{\mathcal{C}^{\alpha',\beta'}} \lesssim \|W\|_{L^{2p',2r'}_tL^{2q',2\widetilde{r}'}_x}^2
\]
holds for all $W \in L^{2p',2r'}_tL^{2q',2\widetilde{r}'}_x(\mathbb{R}^d \times \mathbb{R})$.
\end{proposition}
Here, we are interpreting the function $W$ as an operator which acts by multiplication. Also, we briefly recall that the Schatten space $\mathcal{C}^\alpha=\mathcal{C}^\alpha(L^2(\mathbb{R}^d))$, for $1 \leq \alpha < \infty$, is defined to be the set of all compact operators $\gamma$ on $L^2(\mathbb{R}^d)$ such that the sequence of eigenvalues $(\lambda_j^2)_j$ of $\gamma^*\gamma$ belongs to $\ell^{\alpha/2}(\mathbb{C})$, in which case we define
\[
\|\gamma\|_{\mathcal{C}^\alpha} = \|\lambda\|_{\ell^\alpha} = \bigg(\sum_j |\lambda_j|^\alpha\bigg)^\frac{1}{\alpha}.
\]
When $\alpha = 2$, this coincides with the Hilbert--Schmidt norm and, when the operator is given by an integral kernel, this coincides with the $L^2(\mathbb{R}^d \times \mathbb{R}^d)$ norm of the kernel. Also, when $\alpha = \infty$, we define $\|\gamma\|_{\mathcal{C}^\infty}$ to be the operator norm of $\gamma$ on $L^2(\mathbb{R}^d)$.

The density of the operator $\gamma$ denoted by $\rho_\gamma(x)$ is formally defined as $\rho_\gamma(x)=\gamma(x,x)$ for $x\in\mathbb{R}^d$, where (by a standard abuse of notation) $\gamma(x,y)$ stands for the integral kernel of $\gamma$. For details about the Schatten classes the reader may consult the book by Simon \cite{simon}.

Relevant to our context are operators of the form $\gamma_0 =\sum_j\lambda_j|f_j\rangle\langle f_j|$ associated with a given orthonormal system $(f_j)_j$, where $|f\rangle\langle g|$ is Dirac's notation for the rank-one operator $\phi\mapsto \langle g,\phi\rangle f$. For such $\gamma_0$, we let $\gamma(t)=e^{it\Delta}\gamma_0 e^{-it\Delta}$ for $t\in\mathbb{R}$, and then one may check that
\[
\rho_{\gamma(t)}(x) = \sum_j \lambda_j|e^{it\Delta}f_j(x)|^2. 
\]
This relation connects  Strichartz estimates for orthonormal systems of initial data with the density function $\rho_{\gamma(t)}(x)$. For example, \eqref{e:FS} may be considered in the form
\begin{equation} \label{e:densityform}
\|\rho_{\gamma(t)}(x)\|_{L^p_t L^q_x} \lesssim \|\gamma\|_{\mathcal{C}^{\alpha}}. 
\end{equation}

\subsection{Littlewood--Paley projections and an interpolation method}
We fix a bump function $\varphi\in C^\infty_c(\mathbb{R})$ supported in $[\frac{1}{2},2]$ such that $\sum_{j\in\mathbb{Z}}\varphi(2^{-j}t) = 1$ for all nonzero $t$, and write $(P_j)_{j \in \mathbb{Z}}$ for the family of the Littlewood--Paley projection operators given by
\[
\widehat{P_jf}(\xi)=\varphi(2^{-j}\xi)\widehat{f}(\xi).
\]
The next proposition is the tool we use to extend globally from frequency localised estimates; the price we pay is that such frequency global estimates are in restricted weak-type form. Similar formulations of the same basic idea have appeared several times in the literature (see, for example, \cite{Bourgain} and \cite{Lee}); we need a vector-valued version which we were unable to find elsewhere, so a proof is also provided below.
\begin{proposition}\label{p:Bourgaintrick}
Let $p_0,p_1 > 1$, $q,\alpha_0,\alpha_1 \geq 1$ and $(g_j)_j$ be a uniformly bounded sequence in $L^{2p_i}_tL^{2q}_x$ for each $i=0,1$. If, for each $i=0,1$, there exist $\varepsilon_i>0$ such that
\begin{equation}\label{e:Btrickhyp}
\bigg\| \sum_j
\lambda_j
|P_kg_j|^2
\bigg\|_{L^{p_i,\infty}_tL^{q}_x}
\lesssim
2^{(-1)^{i+1}\varepsilon_ik}
\|\lambda\|_{\ell^{\alpha_i}}
\end{equation}
for all $k \in \mathbb{Z}$, then
\begin{equation*}
\bigg\|
\sum_j
\lambda_j
|g_j|^2
\bigg\|_{L^{p,\infty}_tL^{q}_x}
\lesssim
\|\lambda\|_{\ell^{\alpha,1}}
\end{equation*}
for all sequences $
\lambda=(\lambda_j)_j$ in $\ell^{\alpha,1}(\mathbb{C})
$,
where $\frac{1}{p} = \frac{\theta}{p_0}+\frac{1-\theta}{p_1}$, $\frac{1}{\alpha} = \frac{\theta}{\alpha_0}+\frac{1-\theta}{\alpha_1}$, and 
$
\theta=\frac{\varepsilon_1}{\varepsilon_0+\varepsilon_1}.
$
\end{proposition}
\begin{proof}
It suffices to to consider the characteristic sequence $\lambda_j=\chi_E(j)$, where $E$ is an arbitrary subset of $\mathbb{N}$ such that the cardinality $\# E$ of $E$ is finite; that is, the claimed estimate follows once we show that
\begin{equation} \label{e:Btricksuffices}
|\mathcal{I}_\mu| \lesssim \bigg(\frac{(\# E)^{\frac{1}{\alpha}}}{\mu} \bigg)^p
\end{equation}
for any $\mu > 0$, where
\[
\mathcal{I}_\mu = \bigg\{t\in\mathbb{R}:\bigg\|\sum_{j \in E} |g_j(t,x)|^2\bigg\|_{L^{q}_x}>\mu\bigg\}.
\]
Take any $M\in\mathbb{Z}$, chosen momentarily to optimise the argument. Observe that
\begin{align*}
\bigg|\bigg\{t:\bigg\|\sum_{j \in E} \bigg|\sum_{k \leq M} P_kg_j(x,t)\bigg|^2\bigg\|_{L^{q}_x}>\mu\bigg\}\bigg| & \leq \frac{1}{\mu^{p_1}}\bigg\|\sum_{j \in E}\bigg|\sum_{k\leq M} P_kg_j(x,t)\bigg|^2 \bigg\|_{L^{p_1,\infty}_tL^q_x}^{p_1} \\
& \lesssim \frac{1}{\mu^{p_1}}\bigg(\sum_{k \leq M} \bigg\| \sum_{j \in E} |P_kg_j(x,t)|^2\bigg\|_{L^{p_1,\infty}_tL^q_x}^{\frac{1}{2}} \bigg)^{2p_1}  
\end{align*}
so by \eqref{e:Btrickhyp} with $i=1$ we get
\[
\bigg|\bigg\{t:\bigg\|\sum_{j \in E} \bigg|\sum_{k \leq M} P_kg_j(x,t)\bigg|^2\bigg\|_{L^{q}_x}>\mu\bigg\}\bigg| \lesssim \frac{2^{\varepsilon_1p_1M}}{\mu^{p_1}}
(\# E)^{\frac{p_1}{\alpha_1}}.
\]
We may handle the contribution for $k \geq M+1$ in a similar way using \eqref{e:Btrickhyp} with $i=0$, and since
\[
\bigg\|\sum_{j\in E} |g_j(x,t)|^2 \bigg\|_{L^q_x} \lesssim \bigg\|\sum_{j \in E} \bigg|\sum_{k \leq M} P_kg_j(x,t)\bigg|^2\bigg\|_{L^{q}_x} + \bigg\|\sum_{j \in E} \bigg|\sum_{k \geq M+1} P_kg_j(x,t)\bigg|^2\bigg\|_{L^{q}_x}
\]
for each $t \in \mathbb{R}$, we obtain
\begin{align*}
|\mathcal{I}_\mu| \lesssim \frac{2^{\varepsilon_1p_1M}}{\mu^{p_1}}(\# E)^{\frac{p_1}{\alpha_1}} +\frac{2^{-\varepsilon_0p_0M}}{\mu^{p_0}}(\# E)^{\frac{p_0}{\alpha_0}}.
\end{align*}
Optimising the above estimate in $M$, we see that an elementary computation yields \eqref{e:Btricksuffices}.
\end{proof}

\begin{notation}
We frequently discuss estimates of the form
\begin{equation} \tag{$\mathcal{O}^s((p,r),(q,\tilde{r});(\alpha,\beta))$} 
\bigg\|\sum_j\lambda_j |e^{it\Delta}f_j|^2\bigg\|_{L^{p,r}_tL^{q,\widetilde{r}}_x} \lesssim \|\lambda\|_{\ell^{\alpha,\beta}}
\end{equation}
for all orthonormal systems $(f_j)_j$ in $\dot{H}^s(\mathbb{R}^d)$, and all sequences 
$
\lambda=(\lambda_j)_j$ in $\ell^{\alpha,\beta}(\mathbb{C}),
$
so it will be convenient to introduce the notation $\mathcal{O}^s((p,r),(q,\tilde{r});(\alpha,\beta))$ for this estimate. The parameter $s$ will satisfy the scaling condition $2s=d-(\frac{2}{p}+\frac{d}{q})$.

In this notation, the conjectured estimate \eqref{e:FLLSconj} is $\mathcal{O}^0((p,p),(q,q);(\alpha,1))$ where $d = \frac{2}{p}+\frac{d}{q}$. In such a case, we simplify the notation to $\mathcal{O}^0(p,q;(\alpha,1))$, and we do so in a similar way whenever one or more of the Lorentz spaces reduces to a classical Lebesgue space. For example, $\mathcal{O}^s(p,q;\alpha)$ is the statement that
\begin{equation*}
\bigg\|\sum_j\lambda_j|e^{it\Delta}f_j|^2\bigg\|_{L^{p}_tL^{q}_x}\lesssim\|\lambda\|_{\ell^{\alpha}}
\end{equation*}
holds for all orthonormal systems $(f_j)_j$ in $\dot{H}^s(\mathbb{R}^d)$ and all sequences 
$
\lambda=(\lambda_j)_j$ in $\ell^\alpha(\mathbb{C})
$.
\end{notation}

\section{Frequency localised estimates: Proof of Theorem \ref{t:localmain}} \label{section:local}

In this section, we prove the sufficiency claims in Theorem \ref{t:localmain}. We postpone the justification of the necessary conditions until the proof of the analogous claims in Theorem \ref{t:smooth} (at the end of Section \ref{section:strong}); the orthonormal systems $(f_j)_j$ we use to generate the necessary conditions have frequency support in some fixed annulus, and therefore we use such orthonormal systems for both theorems.

To begin the proof of the sufficiency claims in Theorem \ref{t:localmain}, we begin with the elementary claim that 
\begin{equation} \label{e:localO}
\bigg\|\sum_j\lambda_j |e^{it\Delta}Pf_j|^2\bigg\|_{L^\infty_tL^\infty_x} \lesssim \|\lambda\|_{\ell^\infty}
\end{equation}
holds for all orthonormal systems $(f_j)_j$ in $L^2(\mathbb{R}^d)$. Recall, that $P$ is given by $\widehat{Pf}(\xi) = \phi(\xi)\widehat{f}(\xi)$, where $\phi \in C_c(\mathbb{R}^d)$. The estimate \eqref{e:localO} follows from \cite[Theorem 4]{frank-sabin-1}; here we present a short proof which avoids duality. Firstly, it suffices to consider the case where $\lambda_j=1$ for all $j$. Now fix any $(x,t)\in \mathbb{R}^d\times\mathbb{R}$ and define $\psi_{x,t}(\xi)=e^{-i(x\cdot\xi-t|\xi|^2)}\phi(\xi)$. Then we see from the orthonormality of $(f_j)_j$ in $L^2$ and Bessel's inequality that  
\begin{align*}
\sum_j|e^{it\Delta}Pf_j(x)|^2
&=
\sum_j
|\langle \widehat{\psi_{x,t}},f_j\rangle_{L^2}|^2
\leq
\|\widehat{\psi_{x,t}}\|_{L^2}^2 =\|\phi\|_{L^2}^2
\end{align*} 
from which \eqref{e:localO} clearly follows.

Next, we observe that if $(\frac{1}{q},\frac{1}{p}) \in [B,A)$ and $\alpha = \alpha^*(p,q)$, then 
\begin{equation} \label{e:localBA}
\bigg\|\sum_j\lambda_j |e^{it\Delta}Pf_j|^2\bigg\|_{L^p_tL^q_x} \lesssim \|\lambda\|_{\ell^\alpha}
\end{equation}
holds. Whilst this does not automatically follow from Theorem \ref{t:s=0} (since the orthogonality is not preserved under the action of $P$), the argument given by Frank--Sabin in \cite[Theorem 8]{frank-sabin-1} works just as well with the localisation operator $P$, and hence we simply refer the reader to \cite{frank-sabin-1} for further details.

If $(\frac{1}{q},\frac{1}{p}) \in [C,D]$, then the classical estimate \eqref{e:classical} and the triangle inequality quickly implies
\begin{equation} \label{e:localCD}
\bigg\|\sum_j\lambda_j |e^{it\Delta}Pf_j|^2\bigg\|_{L^p_tL^q_x} \lesssim \|\lambda\|_{\ell^1}.
\end{equation}
In light of \eqref{e:localO}, \eqref{e:localBA} and \eqref{e:localCD}, to prove Theorem \ref{t:localmain}, by complex interpolation it suffices to prove that whenever $d \geq 3$ and $(\frac{1}{q},\frac{1}{p})$ belongs to the line segment $[E,A)$, then
\begin{equation} \label{e:localEA}
\bigg\|\sum_j\lambda_j|e^{it\Delta}Pf_j|^2\bigg\|_{L^{p}_tL^{q}_x}\lesssim\|\lambda\|_{\ell^{p}}
\end{equation}
holds for all orthonormal systems $(f_j)_j$ in $L^2(\mathbb{R}^d)$ and all sequences 
$
\lambda=(\lambda_j)_j$ in $\ell^{p}(\mathbb{C}).
$

Our proof of \eqref{e:localEA} is based on bilinear real interpolation in the spirit of the proof by Keel--Tao \cite{KeelTao} of the endpoint classical Strichartz estimates (at the point $D$). To simplify the notation, we write $\|\cdot\|_{(p,r),q}=\|\cdot\|_{L^{p,r}_tL^q_x}$ and $\|\cdot\|_{p,q}=\|\cdot\|_{L^{p}_tL^q_x}$. Also, since we are using mixed-norm spaces, some care is needed in the use of real interpolation; the precise fact that we shall use is
\begin{equation} \label{e:realintermixed}
(L^{p_0}(L^{q_0}),L^{p_1}(L^{q_1}))_{\theta,p}=L^p(L^{q,p})
\end{equation}
whenever $p_0,p_1,q_0,q_1 \in [1,\infty)$, $\frac{1}{p}=\frac{1-\theta}{p_0}+\frac{\theta}{p_1}$, $\frac{1}{q}=\frac{1-\theta}{q_0}+\frac{\theta}{q_1}$ and $\theta\in(0,1)$ (see, for example, \cite{Lions-Peetre} and \cite{Cwikel74}). We also note that 
\begin{equation} \label{e:realinterseq}
(\ell^{p_0,r_0},\ell^{p_1,r_1})_{\theta,r}=\ell^{p,r}
\end{equation}
whenever $p_0,p_1\in[1,\infty)$, $\frac{1}{p}=\frac{1-\theta}{p_0}+\frac{\theta}{p_1}$, $\theta\in(0,1)$ and all $r_0,r_1,r\in[1,\infty]$. For further details about real interpolation spaces, we refer the reader to \cite{BerghLofstrom} .

\begin{proof}[Proof of \eqref{e:localEA}]
Suppose $\frac{1}{p} = \frac{d}{(d-1)q}$ and $\frac{1}{q} \in [\frac{d-1}{2d},\frac{d-1}{d+1})$. First we use the duality in Proposition \ref{p:dualityprinciple} and re-labelling ($\sigma = p'$ and $r = 2q'$) to re-write the goal \eqref{e:localEA} as
\begin{equation} \label{e:AFlocaldual}
\|W_1T^0W_2\|_{\mathcal{C}^{\sigma}} \lesssim \|W_1\|_{u,r}\|W_2\|_{u,r}
\end{equation}
for $\sigma \in [2,d+1)$, $u = 2\sigma$ and $\frac{2d}{r} = \frac{d-1}{\sigma} + 1$. Here, $T^0$ is the operator given by
\[
T^0F(x,t) = \int_\mathbb{R} e^{i(t-t')\Delta}P^2F(\cdot,t')(x)\, \mathrm{d}t'.
\]
We decompose this operator dyadically $T^0 = \sum_{j \in \mathbb{Z}} T_j$ by writing
\[
T_jF(x,t) = \int_\mathbb{R} \psi_j(t-t') e^{i(t-t')\Delta}P^2F(\cdot,t')(x)\, \mathrm{d}t'
\]
where $\psi_j = \psi(\frac{\cdot}{2^j})$. Here, we choose $\psi \in C^\infty_c(\mathbb{R})$ to be supported in $[\frac{1}{2},2]$, $\sum_{j \in \mathbb{Z}} \psi(\frac{t}{2^j}) = 1$ for $t \neq 0$, and $\widehat{\psi}(0) = 0$; the existence of such a function is easily justified.

Also, we write
\[
\beta(r,s) = \frac{d+1}{2}  - d\bigg(\frac{1}{r} + \frac{1}{s}\bigg).
\]

We establish estimates at $\sigma = 2$ and $\sigma = \infty$, beginning in the former case with the claim
\begin{equation} \label{e:AFC2}
\bigg\|W_1\sum_{j \in \mathbb{Z}} 2^{-\beta(r,r)j} T_jW_2 \bigg\|_{\mathcal{C}^2} \lesssim \|W_1\|_{4,r}\|W_2\|_{4,r}
\end{equation}
for each $r \in (2,4)$. For this, we break-up each $T_j$ one stage further by writing $T_jF(x,t) = T_{j,0}F(x,t) + T_{j,1}F(x,t)$, where
\[
T_{j,1}F(x,t) = \int_\mathbb{R} \psi_j(t-t') e^{i(t-t')\Delta}P^2[\chi_{B(x,C2^j)}F](\cdot,t')(x)\, \mathrm{d}t'
\]
and $B(x,C2^j)$ is the ball centred at $x$ with radius $C2^j$, where the constant $C$ is chosen momentarily to be sufficiently large. 

The contribution from $T_{j,0}$ should be considered as an error term which is more easily handled, so we begin with this part. Observe if $\Psi \in C^\infty(\mathbb{R}^d)$ has compact Fourier support in the ball $B(0,2)$ and $|x| \geq 100|t|$, then by repeated integration by parts we can obtain
\begin{equation} \label{e:nonstatph}
|e^{it\Delta}\Psi(x)| \leq \frac{C_N}{(1 + |x|)^N}
\end{equation}
for any $N \geq 0$. Since the kernel of $W_1T_{j,0}W_2$ at $(x,t,y,t')$ is given by
\[
W_1(x,t)\chi_{B(0,C2^j)^c}(x-y) \psi_j(t-t')e^{i(t-t')\Delta}\Psi(x-y)W_2(y,t')
\]
with $\widehat{\Psi} = \phi^2$, it follows from \eqref{e:nonstatph} (for an appropriate choice of $C$) that
\begin{align*}
& \sum_{j \in \mathbb{Z}} 2^{-\beta(r,r)j}\|W_1T_{j,0}FW_2\|_{\mathcal{C}^2} \\
& \leq C_N \sum_{j \in \mathbb{Z}} \frac{2^{-\beta(r,r)j}}{(1 + 2^j)^N} \bigg(\int_{\mathbb{R}^2}\int_{\mathbb{R}^{2d}}  \frac{|\psi_j(t-t')|^2  |W_1(x,t)|^2|W_2(y,t')|^2}{(1 + |x-y|)^N} \, \mathrm{d}x \mathrm{d}y \mathrm{d}t\mathrm{d}t'\bigg)^{1/2}.
\end{align*}
Young's convolution inequality implies
\begin{align*}
& \sum_{j \in \mathbb{Z}}2^{-\beta(r,r)j} \|W_1T_{j,0}FW_2\|_{\mathcal{C}^2} \\
& \leq C_N \sum_{j \in \mathbb{Z}} \frac{2^{-\beta(r,r)j}}{(1 + 2^j)^N} \bigg(\int_{\mathbb{R}^2}  |\psi_j(t-t')|^2  \|W_1(\cdot,t)\|_r^2 \|W_2(\cdot,t')\|_r^2 \, \mathrm{d}t\mathrm{d}t'\bigg)^{1/2}
\end{align*}
as long as $r \in [2,4]$ and $N$ is sufficiently large, and a further application gives
\begin{align*}
 \sum_{j \in \mathbb{Z}} 2^{-\beta(r,r)j}\|W_1T_{j,0}FW_2\|_{\mathcal{C}^2} & \leq C_N \sum_{j \in \mathbb{Z}} \frac{2^{1-\beta(r,r)j}}{(1 + 2^j)^N} \|W_1\|_{4,r}\|W_2\|_{4,r} 
\end{align*}
and hence the desired estimate
\[
\sum_{j \in \mathbb{Z}} 2^{-\beta(r,r)j} \|W_1T_{j,0}W_2\|_{\mathcal{C}^2} \lesssim \|W_1\|_{4,r}\|W_2\|_{4,r}.
\]

Regarding $T_{j,1}$, we first proceed by using the Cauchy--Schwarz inequality to obtain
\begin{align*}
& \bigg\| \sum_{j\in \mathbb{Z}} 2^{-\beta(r,r)j} W_1T_{j,1}W_2 \bigg\|_{\mathcal{C}^2}^2 \\
& \lesssim \sum_{j \in \mathbb{Z}} 2^{-2\beta(r,r)j} \int_{\mathbb{R}^2}\int_{|x-y| \lesssim 2^j} |\psi_j(t-t')| |W_1(x,t)|^2  |W_2(y,t')|^2 \\
& \qquad \qquad \qquad \qquad \qquad \qquad |e^{i(t-t')\Delta}\Psi(x-y)|^2 \, \mathrm{d}x\mathrm{d}y \mathrm{d}t \mathrm{d}t'
\end{align*}
and to estimate this by $\|W_1\|_{4,r}^2\|W_2\|_{4,r}^2$ for each $r \in (2,4)$, it suffices to prove that
\begin{align} \label{e:Bjgoal}
\sum_{j \in \mathbb{Z}} |\mathcal{T}_{j,r}(V_1,V_2)| \lesssim \|V_1\|_{2,r/2}\|V_2\|_{2,r/2}
\end{align}
for each $r \in (2,4)$, where $\mathcal{T}_{j,r}$ is the bilinear operator given by
\begin{align*}
& \mathcal{T}_{j,r}(V_1,V_2) \\
& := 2^{-2\beta(r,r)j}\int_{\mathbb{R}^2}\int_{|x-y| \lesssim 2^j} |\psi_j(t-t')|  |e^{i(t-t')\Delta}\Psi(x-y)|^2 V_1(x,t) V_2(y,t') \, \mathrm{d}x\mathrm{d}y \mathrm{d}t \mathrm{d}t'.
\end{align*}
We shall establish \eqref{e:Bjgoal} by bilinear interpolation, so we proceed by fixing $r_\ast \in (2,4)$ and establishing a range of asymmetric estimates on each $\mathcal{T}_{j,r_\ast}$. First, we use the dispersive estimate for Schr\"odinger propagator to obtain
\begin{align*}
|\mathcal{T}_{j,r_\ast}(V_1,V_2)| & \lesssim \sum_{j \in \mathbb{Z}} 2^{-2\beta(r_\ast,r_\ast)j}2^{-dj} \times \\
& \int_{\mathbb{R}^2} \int_{|x-y| \lesssim 2^j} |\psi_j(t-t')| |V_1(x,t)|  |V_2(y,t')| \, \mathrm{d}x\mathrm{d}y \mathrm{d}t \mathrm{d}t'
\end{align*}
and then, continuing in a similar manner to the case of $T_{j,0}$, we use Young's convolution inequality to deduce that
\begin{align*}
& |\mathcal{T}_{j,r_\ast}(V_1,V_2)| \\
& \lesssim 2^{-2\beta(r_\ast,r_\ast)j} 2^{(2d(1-\frac{1}{r} - \frac{1}{s})-d)j}\int_{\mathbb{R}^2}  |\psi_j(t-t')| \|V_1(\cdot,t)\|_{r/2}  \|V_2(\cdot,t')\|_{s/2} \, \mathrm{d}t \mathrm{d}t'
\end{align*}
for any $r,s \geq 2$ such that $\frac{1}{r} + \frac{1}{s} \geq \frac{1}{2}$. Another use of Young's convolution inequality allows us to obtain
\begin{align*}
|\mathcal{T}_{j,r_\ast}(V_1,V_2)| \lesssim  2^{\gamma(r,s)j} \|V_1\|_{2,r/2}  \|V_2\|_{2,s/2}
\end{align*}
where
\[
\gamma(r,s) :=  2\beta(r,s)-2\beta(r_\ast,r_\ast) = 2d\bigg(\frac{2}{r_\ast} - \bigg(\frac{1}{r} + \frac{1}{s}\bigg)\bigg).
\]
It follows from these estimates that the vector-valued bilinear operator $\mathcal{T} = \{\mathcal{T}_{j,r_\ast}\}_j$ is bounded between the spaces
\begin{align}
A_0 & \times B_0 \to C_0, \label{e:000}\\
A_0 & \times B_1 \to C_1, \label{e:011}\\
A_1 & \times B_0 \to C_1, \label{e:101}
\end{align}
where
\[ 
A_0 = B_0= L^{2}_t L^{r_0/2}_x, \qquad A_1 = B_1= L^{2}_t L^{r_1/2}_x 
\]
and
\begin{equation} \label{e:C0C1}
C_0 = \ell_{\gamma(r_0,r_0)}^\infty, \qquad C_1= \ell_{\gamma(r_1,r_0)}^\infty
\end{equation}
with $r_0,r_1 \in (2,4)$. Specifically, we choose $r_0, r_1 \in (2,4)$ such that $\frac{1}{r_0} = \frac{1}{r_\ast} + \delta$ and $\frac{1}{r_1} = \frac{1}{r_\ast} - 2\delta$, where $\delta > 0$ is sufficiently small. This choice of $r_0$ and $r_1$ ensures that $\gamma(r_0,r_0) < 0 < \gamma(r_1,r_0) = \gamma(r_0,r_1)$.

In the above, $\| (a_j)_j \|_{\ell^\infty_\beta(X)} = \sup_{j \in \mathbb{Z}} 2^{j\beta}\|a_j\|_X$ and when $p < \infty$ we have 
\[
\| (a_j)_j \|_{\ell^p_\beta(X)} = \bigg(\sum_{j \in \mathbb{Z}} 2^{j\beta}\|a_j\|_X^p\bigg)^{1/p}. 
\]
In this notation, \eqref{e:Bjgoal} will follow once we prove that $\mathcal{T}$ is bounded as follows:
\begin{equation} \label{e:curlyTgoal}
(L^2_tL^{r_0/2},L^2_tL^{r_1/2})_{\frac{1}{3},2} \times (L^2_tL^{r_0/2},L^2_tL^{r_1/2})_{\frac{1}{3},2} \to \ell^1_0.
\end{equation}
Indeed 
$
(L^2_tL^{r_0/2},L^2_tL^{r_1/2})_{\frac{1}{3},2} = L^2_tL^{r_\ast/2,2}
$
and from $r_\ast < 4$ it follows that $L^2_tL^{r_\ast/2}_x$ is embedded in $L^2_tL^{r_\ast/2,2}_x$; hence we obtain \eqref{e:Bjgoal} for all $r = r_\ast \in (2,4)$.

In order to establish \eqref{e:curlyTgoal}, we apply the following bilinear interpolation result; it is a special case of a more general statement which can be found in Bergh--L\"ofstr\"om \cite{BerghLofstrom} (see exercise 5 (b), p. 76). 
\begin{proposition} \label{p:bilinearint}
Suppose $A_0, A_1, B_0, B_1, C_0$ and $C_1$ are Banach spaces, and the bilinear operator $T$ is bounded from $A_0 \times B_0 \to C_0$, $A_0 \times B_1 \to C_1$, and $A_1 \times B_0 \to C_1$. Then $T$ is bounded
\[
(A_0,A_1)_{\frac{\theta}{2},2} \times (B_0,B_1)_{\frac{\theta}{2},2} \to (C_0,C_1)_{\theta,1} 
\]
for all $\theta \in (0,1)$.
\end{proposition}
One can check that by Proposition \ref{p:bilinearint} with $\theta = \frac{2}{3}$ we immediately obtain \eqref{e:curlyTgoal} and this completes our proof of \eqref{e:AFC2}.

From the above argument, it is clear that
\begin{equation} \label{e:AFC2z}
\bigg\|W_1\sum_{j \in \mathbb{Z}} 2^{zj} T_jW_2 \bigg\|_{\mathcal{C}^2} \lesssim \|W_1\|_{4,r}\|W_2\|_{4,r}
\end{equation}
whenever $z \in \mathbb{C}$ is such that ${\rm Re}(z) = -\beta(r,r)$ and $r \in (2,4)$. Note that $\text{Re}\,(z) \in (-\frac{1}{2},\frac{d-1}{2})$ for this range of $r$.

Turning to the case $\sigma = \infty$, we claim that
\begin{equation} \label{e:AFC2z=1main} 
\bigg\|W_1\sum_{j \in \mathbb{Z}} 2^{zj} T_jW_2 \bigg\|_{\mathcal{C}^\infty} \lesssim \|W_1\|_{\infty,\infty}\|W_2\|_{\infty,\infty}
\end{equation}
whenever $z \in \mathbb{C}$ is such that ${\rm Re}(z) = -1$, and for this, it suffices to prove
\begin{equation} \label{e:AFC2z=1}
\bigg\|\sum_{j \in \mathbb{Z}} 2^{zj} T_jF \bigg\|_{L^2_{x,t}} \lesssim \|F\|_{L^2_{x.t}}.
\end{equation}
We have
\[
\widehat{T_jF}(\xi,\tau) = C 2^j \widehat{\psi}(2^j(\tau + |\xi|^2)) \varphi(\xi)^2 \widehat{F}(\xi,\tau)
\]
and since $|\widehat{\psi}(\tau)| \lesssim \min\{|\tau|,|\tau|^{-1}\}$ it follows that 
\[
\sum_{j \in \mathbb{Z}} |\widehat{\psi}(2^j(\tau + |\xi|^2))| \lesssim 1
\]
uniformly in $(\xi,\tau)$. Hence \eqref{e:AFC2z=1} immediately follows when ${\rm Re}(z) = -1$, yielding \eqref{e:AFC2z=1main}.

Finally, we use complex interpolation between the estimates \eqref{e:AFC2z} and \eqref{e:AFC2z=1main}. At $z=0$ this gives the goal \eqref{e:AFlocaldual}, where the range of $\sigma \in [2,d+1)$ arises because the exponent $r$ in \eqref{e:AFC2z} is strictly less than $4$.
\end{proof}

To end this section on the frequency localised estimates, we give some further remarks, firstly, by considering the cases $d=1$ and $d=2$. Since \eqref{e:localO} and \eqref{e:localBA} both hold when $d=1$ and $d=2$ by the same reasoning given above for $d \geq 3$, whenever $(\frac{1}{q},\frac{1}{p})$ belongs to ${\rm int}\,OAB$ and $\alpha = \alpha^*(p,q)$, the estimate \eqref{e:localmain} holds in these dimensions too; the estimates in this region will be used in the proof of Theorem \ref{t:smooth} below. When $d=2$, further (essentially sharp) estimates are available in the remaining region above the critical line $[O,A]$. Indeed, since the endpoint case for the classical estimate \eqref{e:classical} at $C = D$ fails, the estimate \eqref{e:localCD} is not available; however, taking points arbitrarily close to this endpoint we may obtain frequency localised estimates in ${\rm int}\,OCA$ whenever $\alpha < p$.

Our final remark concerns the subregion ${\rm int}\,OEG$, where $G = (0,\frac{1}{2})$. In this case, the frequency localised estimates in Theorem \ref{t:localmain} may be obtained by different means following the argument in Section 3 of \cite{CHP-1}. Indeed, for this subregion, it suffices to consider the line segment $(E,G)$ where $p=2$, in which case the authors considered the dual form \eqref{e:densityform} and capitalised on the fact that rather direct computations can be made on the Schatten $2$-norm of the density.

\section{Strong-type estimates: Proof of Theorem \ref{t:smooth}} \label{section:strong}
\subsection{Sufficiency}
Here, we prove that $\mathcal{O}^s(p,q;\alpha^*(p,q))$ holds whenever $(\frac{1}{q},\frac{1}{p})$ belongs to ${\rm int}\,OAB$ and $2s=d-(\frac{2}{p}+\frac{d}{q})$. Note that this will quickly imply the remaining sufficiency claim in Theorem \ref{t:smooth}(2), namely $\mathcal{O}^s(p,q;\alpha)$ holds for all $\alpha \in [1,p)$ whenever $(\frac{1}{q},\frac{1}{p})$ belongs to ${\rm int}\,OCDA$ and $2s=d-(\frac{2}{p}+\frac{d}{q})$. For $d \geq 3$, this follows by complex interpolation between estimates $\mathcal{O}^s(p,q;\alpha^*(p,q))$ in ${\rm int}\,OAB$ with the estimates $\mathcal{O}^s(p,q;1)$ on the line segment $(C,D)$; note that $\alpha^*(p,q) = p$ for $(\frac{1}{q},\frac{1}{p})$ on the line segment $(O,A)$, so we obtain the claim by interpolating with points in the interior of $OAB$ arbitrarily close to the line $(O,A)$. When $d=2$ the points $C$ and $D$ coincide and at this point the classical (single function) Strichartz estimate \eqref{e:classical} fails. So, we interpolate points belonging to ${\rm int}\,OCA$ arbitrarily close to the point $C$ and points belonging to ${\rm int}\,OAB$ arbitrarily close to the line $(O,A)$. Such arguments were used in \cite{frank-sabin-2} to prove Theorem \ref{t:FS2}.

The rest of the section is used to prove that $\mathcal{O}^s(p,q;\alpha^*(p,q))$ holds whenever $(\frac{1}{q},\frac{1}{p})$ belongs to ${\rm int}\,OAB$ and $2s=d-(\frac{2}{p}+\frac{d}{q})$. There are several stages to the proof. First, we combine the frequency localised estimates in Theorem \ref{t:localmain} with the globalisation argument in Proposition \ref{p:Bourgaintrick} to obtain restricted weak-type estimates. Next, we observe that a refinement of Theorem \ref{t:s=0} is possible on the line segment $(A,F)$, where $L^p_t$ is replaced by the Lorentz space $L^{p,r}_t$ for an appropriate $r < p$, and this allows us to upgrade our restricted weak-type estimates to restricted strong-type estimates in the interior of $OFA$. Finally, an argument using real interpolation yields the desired estimates $\mathcal{O}^s(p,q;\alpha^*(p,q))$ on ${\rm int}\,OFA$, which, by further use of complex interpolation, yields $\mathcal{O}^s(p,q;\alpha^*(p,q))$ on ${\rm int}\,OAB$.

\subsection*{Restricted weak-type estimates}
First we claim that whenever $d \geq 1$, $s\in\mathbb{R}$, $k \in \mathbb{Z}$, $(\frac{1}{q},\frac{1}{p})$ belongs to $OAB \setminus [A,O)$ and $\alpha=\alpha^*(p,q)$, then
\begin{equation}\label{e:localOAB}
\bigg\|
\sum_j\lambda_j|e^{it\Delta}P_k |D|^{-s}f_j|^2
\bigg\|_{L^p_tL^q_x}
\lesssim
2^{k(d-2s-\frac{2}{p}-\frac{d}{q})}
\|\lambda\|_{\ell^\alpha}
\end{equation}
holds for all orthonormal systems $(f_j)_j$ in $L^2(\mathbb{R}^d)$ and all sequences 
$
\lambda=(\lambda_j)_j$ in $\ell^\alpha(\mathbb{C}).
$
This follows immediately from Theorem \ref{t:localmain} (and the remarks at the end of Section \ref{section:local}) when $k=0$, and the case of general $k \in \mathbb{Z}$ is subsequently obtained by a rescaling argument.

Using Proposition \ref{p:Bourgaintrick}, we may upgrade the estimates in \eqref{e:localOAB}.
\begin{proposition} \label{p:afterBtrick}
Suppose $d \geq 1$ and let $(\frac{1}{q},\frac{1}{p})$ belong to ${\rm int}\,OAB$. If $2s=d-(\frac{2}{p}+\frac{d}{q})$ and $\alpha = \alpha^*(p,q)$, then 
\begin{equation} \label{e:RWTinOAB}
\bigg\|\sum_j\lambda_j |e^{it\Delta}f_j|^2\bigg\|_{L^{p,\infty}_tL^{q}_x} \lesssim \|\lambda\|_{\ell^{\alpha,1}}
\end{equation}
holds for all orthonormal sequences $(f_j)_j$ in $\dot{H}^s(\mathbb{R}^d)$ and all sequences 
$
\lambda=(\lambda_j)_j$ in $\ell^{\alpha,1}(\mathbb{C}).
$
\end{proposition}
\begin{proof}
Let $(\frac{1}{q},\frac{1}{p})$ be an arbitrary point from ${\rm int}\,OAB$, and choose $\delta > 0$ sufficiently small so that $(\frac{1}{q},\frac{1}{p_0})$ and $(\frac{1}{q},\frac{1}{p_1})$ also belong to ${\rm int}\,OAB$, where $\frac{1}{p_0} = \frac{1}{p} + \delta$ and $\frac{1}{p_1} = \frac{1}{p} - \delta$. Next, observe that if we define
\[
\varepsilon_i = (-1)^{i}\bigg(\frac{2}{p_0} + \frac{d}{q} - d + 2s\bigg)
\]
for $i=0,1$, then \eqref{e:localOAB} implies
\begin{equation*}
\bigg\|
\sum_j\lambda_j|e^{it\Delta}P_k |D|^{-s}f_j|^2
\bigg\|_{L^{p_i,\infty}_tL^q_x}
\lesssim
2^{(-1)^{i+1}\varepsilon_i}
\|\lambda\|_{\ell^{\alpha_i}}
\end{equation*}
holds for all orthonormal systems $(f_j)_j$ in $L^2(\mathbb{R}^d)$, where $\alpha_i = \alpha^*(p_i,q)$. 

Since $\frac{2}{p} + \frac{d}{q} = d-2s$, our choice of $p_0$ and $p_1$ means that $\varepsilon_0 = \varepsilon_1$ and thus, from Proposition \ref{p:Bourgaintrick}, we immediately obtain
\begin{equation*}
\bigg\|
\sum_j\lambda_j|e^{it\Delta}|D|^{-s}f_j|^2
\bigg\|_{L^{p,\infty}_tL^{q}_x}
\lesssim
\|\lambda\|_{\ell^{\alpha^*(p,q),1}}
\end{equation*}
for all orthonormal systems $(f_j)_j$ in $L^2(\mathbb{R}^d)$, as required.
\end{proof}

\subsection*{Refinement of Theorem \ref{t:s=0}}
At a key stage in the proof of Theorem \ref{t:s=0} in \cite{frank-sabin-1}, the Hardy--Littlewood--Sobolev inequality is used; as we shall see below, we may simply invoke the optimal Lorentz space refinement of the Hardy--Littlewood--Sobolev inequality (due to O'Neil \cite{ONeil}) to obtain the following.
\begin{proposition}\label{p:s=0Lorentz}
Suppose $d \geq 1$ and let $(\frac{1}{q},\frac{1}{p})$ belong to $(A,F)$. If $\alpha=\alpha^*(p,q)$, then
\[
\bigg\|\sum_j\lambda_j |e^{it\Delta}f_j|^2\bigg\|_{L^{p,\alpha}_tL^{q}_x} \lesssim \|\lambda\|_{\ell^{\alpha}}
\]
holds for all orthonormal sequences $(f_j)_j$ in $L^2(\mathbb{R}^d)$ and all sequences 
$
\lambda=(\lambda_j)_j$ in $\ell^{\alpha}(\mathbb{C}).
$
\end{proposition}
\begin{proof}
As we have already indicated, we can prove Proposition \ref{p:s=0Lorentz} by following very closely the proof of Theorem \ref{t:s=0} in \cite{frank-sabin-1} and making one small adjustment using the Lorentz refinement  of the Hardy--Littlewood--Sobolev inequality of O'Neil \cite{ONeil}. In one dimension, this states that
\begin{equation} \label{e:LorentzHLS}
\bigg|\int_{\mathbb{R}} \int_{\mathbb{R}} \frac{g_1(t_1)g_2(t_2)}{|t_1-t_2|^\sigma} \,\mathrm{d}t_1\mathrm{d}t_2\bigg| \lesssim \|g_1\|_{L^{p_1,r_1}}\|g_2\|_{L^{p_2,r_2}}
\end{equation}
where $\sigma \in (0,1)$, $p_j \in (1,\infty)$ are such that $\frac{1}{p_1} + \frac{1}{p_2} + \beta = 2$, and $\frac{1}{r_1} + \frac{1}{r_2} \geq 1$. Given the similarity to the argument in \cite{frank-sabin-1}, we omit some details and point the reader to \cite{frank-sabin-1} for full details.

By Proposition \ref{p:dualityprinciple}, the desired estimate holds if and only if 
\begin{equation}\label{e:s=0Lorentzdual}
\|WUU^*\overline{W}\|_{\mathcal{C}^{\alpha'}} \lesssim \|W\|_{L^{2p',2\alpha'}_tL^{2q'}_x}^2
\end{equation}
for any $W\in L^{2p',2\alpha'}_tL^{2q'}_x$, where $Uf(x,t) = e^{it\Delta}f(x)$. To obtain \eqref{e:s=0Lorentzdual}, we consider the family of operators 
$T_z$ defined by 
\[
\widehat{T_z\phi}(\tau,\xi)
= \frac{1}{\Gamma(z+1)}
(\tau-|\xi|^2)^z_+ \widehat{\phi}(\tau,\xi),
\]
for $(\tau,\xi)\in\mathbb{R}\times\mathbb{R}^d$ and $\frac{d+1}{2}<-{\rm Re}(z)<\frac{d+2}{2}$; note that $T_{-1}=UU^*$. For this family of operators, by using \eqref{e:LorentzHLS} we obtain 
\[
\left\|
W_1T_{z}W_2
\right\|_{\mathcal{C}^2}^2
\leq
C({\rm Im}(z))
\|
W_1
\|_{L^{\widetilde{p},4}_tL^{2}_x}^2
\|
W_2
\|_{L^{\widetilde{p},4}_tL^{2}_x}^2,
\]
where $\frac{2}{\widetilde{p}}=-{\rm Re}(z)-\frac{d}{2}$ and $\frac{d+1}{2}<-{\rm Re}(z)<\frac{d+2}{2}$. Here, $C({\rm Im}(z))$ is a constant which grows exponentially with ${\rm Im}(z)$.

On the other hand, if ${\rm Re}(z)=0$, then from Plancherel's theorem it follows that 
\[
\left\|
W_1T_{z}W_2
\right\|_{\mathcal{C}^\infty}
\leq
C({\rm Im}(z))
\|
W_1
\|_{L^{\infty}_tL^{\infty}_x}
\|
W_2
\|_{L^{\infty}_tL^{\infty}_x}.
\] 
By complex interpolation, we obtain the desired inequality \eqref{e:s=0Lorentzdual} with $\alpha=\alpha^*(p,q)$.
\end{proof}

\subsection*{Restricted weak-type to restricted strong-type}

\begin{proposition} \label{p:afterLorentzFS}
Suppose $d \geq 1$ and let $(\frac{1}{q},\frac{1}{p})$ belong to ${\rm int}\,OFA$. If $2s=d-(\frac{2}{p}+\frac{d}{q})$ and $\alpha = \alpha^*(p,q)$, then 
\[
\bigg\|\sum_j\lambda_j |e^{it\Delta}f_j|^2\bigg\|_{L^{p}_tL^{q}_x} \lesssim \|\lambda\|_{\ell^{\alpha,1}}
\]
holds for all orthonormal sequences $(f_j)_j$ in $\dot{H}^s(\mathbb{R}^d)$ and all sequences 
$
\lambda=(\lambda_j)_j$ in $\ell^{\alpha,1}(\mathbb{C}).
$
\end{proposition}
\begin{proof}
We fix $(\frac{1}{q},\frac{1}{p})$ in ${\rm int}\,OFA$ and let $(\frac{1}{q_0},\frac{1}{p_0})$ be the intersection point of the line through the origin and $(\frac{1}{q},\frac{1}{p})$ with the line segment $(A,F)$. Next we define $\varepsilon_0 = \frac{1}{2}(\frac{d+1}{q_0} - (d-1)) \in (0,\frac{1}{d+2})$ and
$
\theta = 1 - \frac{1}{p}(\frac{1}{p_0} + \varepsilon_0)^{-1} \in (0,1).
$
Finally, we define $(\frac{1}{p_1},\frac{1}{q_1})$ by $\frac{1}{p} = \frac{1-\theta}{p_0} + \frac{\theta}{p_1}$ and $\frac{1}{q} = \frac{1-\theta}{q_0} + \frac{\theta}{q_1}$. 

One can check that $(\frac{1}{p_1},\frac{1}{q_1})$ belongs to ${\rm int}\,OAB$, and therefore an application of Proposition \ref{p:afterBtrick} gives $\mathcal{O}^{s_1}((p_1,\infty),q_1;(\alpha_1,1))$, where $\alpha_1 =\alpha^*(p_1,q_1)$ and $2s_1 = d-(\frac{2}{p_1}+\frac{d}{q_1})$. Also, we have $\mathcal{O}^0((p_0,r_0),q_0;\alpha_0)$, where $\alpha_0 = r_0 = \alpha^*(p_0,q_0)$ and $\frac{2}{p_0}+\frac{d}{q_0} = d$, by Proposition \ref{p:s=0Lorentz}. We claim that using complex interpolation between these two estimates with $\theta$ above gives the estimate $\mathcal{O}^s(p,q;(\alpha^*(p,q),\beta))$, for some $\beta \geq 1$ and where $2s = d-(\frac{2}{p}+\frac{d}{q})$ (hence slightly stronger than the desired estimate). Indeed, it is clear that $2s = 2s_1\theta = d - (\frac{2}{p} + \frac{d}{q})$ and $\alpha$ given by $\frac{1}{\alpha} = \frac{1-\theta}{\alpha_0} + \frac{\theta}{\alpha_1}$ coincides with $\alpha^*(p,q)$. Finally, a computation shows that if $r$ is given by $\frac{1}{r} = \frac{1-\theta}{r_0}$, then $r = p$.
\end{proof}

\subsection*{Strong type estimates}
At this final step of the proof, we use real interpolation to upgrade the restricted strong-type estimates in Proposition \ref{p:afterLorentzFS} to the desired strong-type estimates $\mathcal{O}^s(p,q;\alpha^*(p,q))$. 

\begin{proof}[Proof of Theorem \ref{t:smooth}: Sufficiency]
Firstly, we observe that it is sufficient to show the desired estimate $\mathcal{O}^s(p,q;\alpha^*(p,q))$ for all $(\frac{1}{q},\frac{1}{p})$ belonging to ${\rm int}\,OFA$. Indeed, once this is established, we may employ complex interpolation once again with the estimates in Theorem \ref{t:s=0} on the line segment $[B,A)$ to obtain $\mathcal{O}^s(p,q;\alpha^*(p,q))$ for all $(\frac{1}{q},\frac{1}{p})$ belonging to ${\rm int}\,OAB$.

As we shall soon see, the advantage of first considering the region $OFA$ is that 
\begin{equation}\label{e:lucky}
\alpha^*(p,q)\leq p\leq q
\end{equation}
whenever $(\frac{1}{q},\frac{1}{p})$ belongs to $OFA$. Indeed, one easily sees that $\alpha^*(p,q) \leq p$ is equivalent to $\frac{1}{p} \leq \frac{d}{(d-1)q}$ which means below the line $[O,A]$ (and $p \leq q$ obviously holds in $OFA$ since $F$ lies on the diagonal $\frac{1}{q} = \frac{1}{p}$).

Since we wish to use real interpolation, we fix $s\in(0,\frac{d}{2})$ and take any two points $(\frac{1}{q_i},\frac{1}{p_i})$ from ${\rm int}\,OFA$ such that $\frac{2}{p_i}+\frac{d}{q_i} = d - 2s$ for $i=0,1$. From Proposition \ref{p:afterLorentzFS} we know that $\mathcal{O}^s(p_i,q_i;(\alpha^*(p_i,q_i),1))$ holds for $i = 0,1$. This means that if we fix an orthonormal system $(f_j)_j$ in the common space $\dot{H}^s$, then real interpolation, \eqref{e:realintermixed} and \eqref{e:realinterseq} yield
\begin{equation} \label{e:afterrealiner}
\bigg\|\sum_j\lambda_j|e^{it\Delta}f_j|^2 \bigg\|_{L^{p}_tL^{q,p}_x} \lesssim \|\lambda\|_{\ell^{\alpha^*(p,q),p}}
\end{equation}
with $\frac{1}{p}=\frac{1-\theta}{p_0}+\frac{\theta}{p_1}$, $\frac{1}{q}=\frac{1-\theta}{q_0}+\frac{\theta}{q_1}$ and any $\theta\in(0,1)$; that is, \eqref{e:afterrealiner} holds for all $(\frac{1}{q},\frac{1}{p})$ belonging to ${\rm int}\,OFA$. Thanks to \eqref{e:lucky}, we may deduce from the nesting of Lorentz spaces that
\begin{equation*}
\bigg\|\sum_j\lambda_j|e^{it\Delta}f_j|^2 \bigg\|_{L^{p}_tL^{q}_x} \lesssim \|\lambda\|_{\ell^{\alpha^*(p,q)}}
\end{equation*}
and therefore $\mathcal{O}^s(p,q;\alpha^*(p,q))$ holds for all $(\frac{1}{q},\frac{1}{p})$ belonging to ${\rm int}\,OFA$ with $2s = \frac{2}{p}+\frac{d}{q}$, as claimed.
\end{proof}

\subsection{Necessity} 
Here, we explicitly construct two types of orthonormal systems $(f_j)_j$ to prove the necessity claims in Theorem \ref{t:smooth}; i.e. if $\mathcal{O}^s(p,q;\alpha)$ holds, then necessarily $\alpha \leq \alpha^*(p,q)$ and $\alpha \leq p$.

\subsection*{Necessity of $\alpha \leq \alpha^*(p,q)$}
Let $R\gg 1$, $\chi\in C_c^\infty(B(0,\frac12))$, and $v\in  R^{-1} \mathbb Z^d\cap    [2,4)^d$.
For each such $v$, define $f_v$ by 
\[ \widehat{f_v}(\xi) = \chi_v(\xi)=R^\frac d2 \chi(R(\xi-v)). \] 
A simple computation shows that 
\begin{equation}
\label{e:packet}  |e^{it\Delta} f_v(x)|\gtrsim  R^{-\frac d2} \chi_{T_v}(x,t),  
\end{equation}
with the implicit constant independent of $v, R$, where $T_v$ is the tube given by
\[ 
T_v=\{(x,t): |x-2tv| \le cR, \,\, |t|\le cR^2  \} 
\] 
and $c$ is a sufficiently small number.  In fact, by a change of variables we see that 
\[  
|e^{it\Delta}f_v(x)| = CR^{-\frac{d}{2}} \bigg| \int_{\mathbb{R}^d} e^{iR^{-1}\xi\cdot(x-2tv)} e^{it|R^{-1}\xi|^2}\chi(\xi)\,\mathrm{d}\xi \bigg|.
\] 
Since $R^{-1}\xi\cdot(x-2tv) = O(c)$ and $t|R^{-1}\xi|^2=O(c)$ if $(x,t) \in T_v$ and $|\xi| \leq \frac{1}{2}$, we see that 
$|e^{it\Delta} f_v(x)|\sim R^{-d/2}$ for such $(x,t)$; hence, \eqref{e:packet} follows.  

Then, since $(f_v)_v$ have disjoint Fourier supports and $\|f_v\|_{\dot{H}^s}\sim1$, if $\mathcal{O}^s(p,q;\alpha)$ holds then  
\[   
R^{- d}\bigg\|  \sum_{v}  \chi_{T_v} \bigg\|_{L^p_t(-cR^2,cR^2); L^q_x(B(0, cR)) } \lesssim  \bigg(\sum_{v} 1 \bigg)^\frac1\alpha \sim R^{\frac d\alpha}.    
\] 
Note that all ${T_v}$ contain $(-\tilde cR,cR)\times B(0, \tilde cR)$ with a small enough choice of $\tilde c$.  Hence, the above gives $R^{\frac 1p+\frac dq}\lesssim R^\frac d\alpha$.   Letting $R\to \infty$ we obtain the necessary condition $\frac1p+\frac dq\le \frac d\alpha$; i.e. $\alpha\leq \alpha^*(p,q)$. \qed

\subsection*{Necessity of $\alpha \leq p$}
Fix the function $g$ given by
\[
\widehat{g}(\xi) = C_0\frac{\chi_{[\sqrt{2\pi},\sqrt{4\pi}]}(|\xi|)}{|\xi|^\sigma}
\]
where $\sigma = \frac{d-2+2s}{2}$ and the constant $C_0$ is such that $C_0^2 = \frac{|\mathbb{S}^{d-1}|}{2(2\pi)^{d-1}}$. Then we define $f_j = e^{ij\Delta}g$. A straightforward calculation using polar coordinates reveals that
\[
(f_j,f_k)_{\dot{H}^s} = \frac{1}{(2\pi)^d }\int_{\mathbb{R}^d} |\widehat{g}(\xi)|^2 e^{-i(j-k)|\xi|^2}|\xi|^{2s} \, \mathrm{d}\xi = \delta_{jk}
\]
so that $(f_j)_j$ forms an orthonormal system in $\dot{H}^s(\mathbb{R}^d)$.

For any $\varepsilon \in (0,1)$ and nonnegative sequence $\lambda$, we observe that 
\begin{align*}
\bigg\|\sum_j\lambda_j|e^{it\Delta}f_j|^2\bigg\|_{L^p_tL^q_x}^p
&\geq
\sum_{n\in\mathbb{Z}}\int_n^{n+\varepsilon}\bigg(\int_{\mathbb{R}^d}\bigg[\sum_j\lambda_j|e^{i(t-j)\Delta}g(x)|^2 \, \mathrm{d}x \bigg]^q\bigg)^\frac pq \mathrm{d}t\\
&\geq \sum_{n\in\mathbb{Z}}\lambda_n^p\int_n^{n+\varepsilon}\|e^{i(t-n)\Delta}g\|_{L^{2q}_x}^{2p}\,\mathrm{d}t
\end{align*}
and so, by choosing $\varepsilon > 0$ suitably small (depending on $d$, $s$ and $q$) it follows that
\begin{align*}
\bigg\|\sum_j\lambda_j|e^{it\Delta}f_j|^2\bigg\|_{p,q}^p \gtrsim \|\lambda\|_{\ell^{p}}^p.
\end{align*}
Thus, if we assume that $\mathcal{O}^s(p,q;\alpha)$ holds, then the above example leads to $\|\lambda\|_{\ell^{p}} \lesssim \|\lambda\|_{\ell^{\alpha}}$ for all nonnegative sequences $\lambda$; hence $\alpha \leq p$, as claimed. \qed


Finally, we return to the claimed necessary conditions in Theorem \ref{t:localmain}. If $|\phi(\xi)| \gtrsim 1$ on some annulus, it is clear that  slight modification of  two systems of initial data used above  generates the same necessary  conditions (with re-scaling to align the support with the annulus where $\phi$ is bounded below) since these systems have frequency supports in some fixed annuli.  However, the same systems (especially the second one)  does not generally work  if  we no longer have $|\phi(\xi)| \gtrsim 1$ on a certain  annulus. Nevertheless, this can be simply overcome by using translation in the frequency side. In fact, since $\phi$ is nontrivial, there is a $v\in \mathbb R^d$ with $\phi(v)\neq 0$. The transformation $\xi\to \xi+v$ changes
\[ \phi\to \widetilde \phi= \phi(\cdot+v), \   f_j\to \widetilde f_j= e^{-iv\cdot  x}f_j,  
\  |e^{it\Delta}P f_j(x) |\to  |e^{it\Delta}\widetilde P \widetilde f_j(x-2vt)|.\] 
Here, $\widetilde P$ is the projection operator given by $\widetilde \phi$.  It follows that 
$\|\sum_j\lambda_j|e^{it\Delta}Pf_j|^2\|_{L^p_tL^q_x}$ $=$ $\|\sum_j\lambda_j|e^{it\Delta}\widetilde P \widetilde f_j|^2\|_{L^p_tL^q_x}$ 
while $(\widetilde f_j)_j$ remains orthonormal if so is  $(f_j)_j$.  Now,  since $|\widetilde \phi|\sim 1$ on a ball (and on an annulus),  we can use the previous two systems of initial data  to get the necessary  conditions in   Theorem \ref{t:localmain}.

\section{The critical line $[O,A]$ and Conjecture \ref{conj:FLLS}} \label{section:OA}

In this section, we prove two negative results. The first is that, without the frequency localisation, the strong-type estimates $\mathcal{O}^s(p,q;p)$ fail on $[O,A]$; in fact, we prove the stronger statement that $\mathcal{O}^s((p,\infty),(q,\infty);(p,r))$ fails \emph{for all} $r > 1$ when $(\frac{1}{q},\frac{1}{p})$ belongs to $[O,A]$. This observation draws attention to the conjectured estimate $\mathcal{O}^s(p,q;(p,1))$ in Conjecture \ref{conj:FLLS} where $(p,q) = (\frac{d+1}{d},\frac{d+1}{d-1})$ and $r=1$. Our second negative result is to prove that this conjecture fails when $d=1$; currently, we do not have an indication of whether the conjecture is true or not for $d \geq 2$ and we believe this is a very interesting open problem.

\subsection{Semi-classical limiting argument} In order to prove the negative results on $[O,A]$, we show that certain induced estimates for the velocity average of the kinetic transport equation fail on $[O,A]$. This requires us to first make the following observation based on a semi-classical limiting argument.
\begin{proposition} \label{p:semiclassical}
Let $p,q\in[1,\infty]$, $r,\widetilde{r},\beta\in[1,\infty)$ and $s\in{[0,\frac{d}{2})}$ be such that $2s = d - (\frac{2}{p} + \frac{d}{q})$. If $\alpha = \alpha^*(p,q)$ and
\begin{equation}\label{e:weakRHS}
\bigg\|\sum_j\lambda_j|e^{it\Delta}f_j|^2\bigg\|_{L^{p,r}_tL^{q,\widetilde{r}}_x}
\lesssim \|\lambda\|_{\ell^{\alpha,\beta}}
\end{equation}
holds for all orthonormal systems $(f_j)_j$ in $\dot{H}^s(\mathbb{R}^d)$ and all sequences 
$
\lambda=(\lambda_j)_j$ in $\ell^{\alpha,\beta}(\mathbb{C}),
$
then 
\begin{equation}\label{e:weakRHS_KT}
\bigg\|\int_{\mathbb{R}^d}f(x-tv,v) \,\frac{\mathrm{d}v}{|v|^{2s}}\bigg\|_{L^{p,r}_tL^{q,\widetilde{r}}_x}\lesssim \|f\|_{L^{\alpha,\beta}}
\end{equation}
whenever $f \in L^{\alpha,\beta}$.
\end{proposition}
Before giving a proof of Proposition \ref{p:semiclassical}, we would like to make some remarks. The function $F(x,v,t) = f(x-tv,v)$ satisfies the kinetic transport equation
\[
(\partial_t  + v \cdot \nabla_x)F(x,v,t) = 0, \qquad F(x,v,0) = f(x,v)
\]
for $(x,v,t) \in \mathbb{R}^d \times \mathbb{R}^d \times \mathbb{R}$, and 
\begin{equation} \label{e:rhodefn}
\rho f(x,t) = \int_{\mathbb{R}^d}f(x-tv,v) \,\mathrm{d}v
\end{equation}
is the velocity average of the solution. Estimates of the form \eqref{e:weakRHS_KT} are typically referred to as Strichartz estimates for the kinetic transport equation; the classical case is $s=0$ and $r = \alpha^*(p,q)$, in which case it is known that \eqref{e:weakRHS_KT} holds if and only if $(\frac{1}{q},\frac{1}{p})$ belongs to $[B,A)$. The positive results were obtained in \cite{CastellaPerthame} and \cite{KeelTao}, and the failure at the endpoint $A$ was shown in \cite{BBGL_CPDE} for all $d \geq 1$. The argument establishing failure at the endpoint $A$ profitably used duality and it was shown that
\begin{equation} \label{e:s=0dual}
\bigg\| \int_{\mathbb{R}} g(x+tv,t)\,\mathrm{d}t \bigg\|_{L^{d+1}} \lesssim \|g\|_{L^{d+1}_tL^{\frac{d+1}{2}}_x}
\end{equation}
fails by multiplying out the norm on the left-hand side and ultimately testing on smooth and rapidly decaying $g$ whose Fourier transform is non-zero at the origin. A simplification was given in \cite{BLNS} by showing the left-hand side of \eqref{e:s=0dual} is infinite on the centred gaussian $g(x,t) = e^{-\pi(t^2 + |x|^2)}$, and such an explicit argument did not rely on the fortuitous fact that Lebesgue exponent coincides with an integer.

We also note that the failure of \eqref{e:weakRHS_KT} at $A$ when $s=0$ and $r = \alpha^*(p,q)$ was shown when $d=1$ prior to \cite{BBGL_CPDE} by Guo--Peng \cite{GuoPeng} and Ovcharov \cite{Ovcharov_Endpoint}. The argument of Ovcharov used characteristic functions of Besicovitch (or Kakeya) sets; these are sets containing a unit line segment in all possible directions and a famous argument of Besicovitch generates such sets with arbitrarily small measure; this argument is particularly relevant to the present discussion and will be used to disprove Conjecture \ref{conj:FLLS} when $d=1$.

The connection between solutions of the free Schr\"odinger equation and the kinetic transport equation is well documented and proceeds by a semi-classical limiting argument. Thus, we are not viewing Proposition \ref{p:semiclassical} as particularly novel and we present its statement and proof below for completeness and since we were not able to find elsewhere in the literature the statement in the form that we need. Sabin presented the special case $r = \alpha^*$ in Lemma 9 of \cite{sabin} and we use a similar argument to extend his observations to the setting of Lorentz spaces.
\begin{proof}[Proof of Proposition \ref{p:semiclassical}] 
First we note that \eqref{e:weakRHS} implies
\begin{equation}\label{e:weakRHS'}
\|\rho_{\gamma(t)}(x)\|_{L^{p,r}_t L^{q,\widetilde{r}}_x} \lesssim \|\gamma_0\|_{\mathcal{C}^{\alpha,\beta}(L^2)}, 
\end{equation}
where $\gamma(t)=|D|^{-s}e^{it\Delta}\gamma_0 e^{-it\Delta} |D|^{-s}$. Next, we fix any $f$ in the Schwartz class $\mathcal{S}(\mathbb{R}^d\times\mathbb{R}^d)$ and test \eqref{e:weakRHS'} on the semi-classical Weyl quantisation $\gamma_0 = \gamma_0(f;h)$ of $f$, whose kernel is given by
\[
\gamma_0(x,x') = \int_{\mathbb{R}^d} f(\tfrac{x+x'}{2},v)e^{i\frac{(x-x')\cdot v}{h}}\, \mathrm{d}v.
\]
The parameter $h$ will later be sent to zero. The Fourier transform of $\gamma_0$ on $\mathbb{R}^d \times \mathbb{R}^d$ is given by
\begin{equation} \label{e:gammaFT}
\widehat{\gamma_0}(v,v') = (2\pi h)^d \mathcal{F}_x f(\cdot,\tfrac{h}{2}(v-v'))(v+v'),
\end{equation}
where $\mathcal{F}_x$ denotes the Fourier transform in the $x$ variable.

A direct computation, making use of \eqref{e:gammaFT}, reveals that if
\[
\widetilde{\gamma}(t) = h^{-2s} \gamma(ht),
\]
then
\begin{align*}
\rho_{\widetilde{\gamma}(t)}(x) & = \frac{h^{-2s}}{(2\pi)^{2d}} \int_{\mathbb{R}^{2d}} e^{-i\frac{th}{2}|v-v'|^2} e^{i\frac{th}{2}|v'|^2} |v-v'|^{-s}|v'|^{-s} \widehat{\gamma_0}(v-v',v') \, \mathrm{d}v' \, e^{ix\cdot v} \mathrm{d}v \\
& = \frac{h^{d-2s}}{(2\pi)^{d}} \int_{\mathbb{R}^{2d}} e^{-i\frac{th}{2}(v \cdot (v - 2v'))}  |v-v'|^{-s}|v'|^{-s} \mathcal{F}_x f(\cdot,\tfrac{h}{2}(v-2v'))(v) \, \mathrm{d}v' \, e^{ix\cdot v}\mathrm{d}v 
\end{align*}
and therefore, by a change of variables,
\begin{align*}
\rho_{\widetilde{\gamma}(t)}(x) = \frac{1}{(2\pi)^{d}} \int_{\mathbb{R}^{2d}} e^{-itv \cdot v''}  |v'' + \tfrac{h}{2}v|^{-s}|v'' - \tfrac{h}{2}v|^{-s} \mathcal{F}_x f(\cdot,v'')(v) \, \mathrm{d}v'' \, e^{ix\cdot v}\mathrm{d}v. 
\end{align*}
It follows that
\[
\|\rho_{\widetilde{\gamma}(t)}(x)\|_{L^{p,r}_t L^{q,\widetilde{r}}_x} \to \bigg\| \int_{\mathbb{R}^d} f(x-tv'',v'') |v''|^{-2s} \, \mathrm{d}v''\bigg\|_{L^{p,r}_t L^{q,\widetilde{r}}_x}
\]
as $h \to 0$, or equivalently,
\begin{equation} \label{e:LHSsemi}
h^{-(\frac{1}{p}+2s)}\|\rho_{\gamma(t)}(x)\|_{L^{p,r}_t L^{q,\widetilde{r}}_x} \to \bigg\| \int_{\mathbb{R}^d} f(x-tv'',v'') |v''|^{-2s} \, \mathrm{d}v''\bigg\|_{L^{p,r}_t L^{q,\widetilde{r}}_x}
\end{equation}
as $h \to 0$.

For the right-hand side, we observe that
\begin{equation} \label{e:RHSsemi}
\|\gamma_0\|_{\mathcal{C}^{\alpha,\beta}} \lesssim h^{d - \frac{d}{\alpha}} \sum_{0 \leq |n| + |m| \leq 2d + 2} h^{\frac{|n| + |m|}{2}} \|\partial_x^n \partial_v^m f\|_{L^{\alpha,\beta}_{x,v}}
\end{equation}
for $\alpha,\beta\in[1,\infty)$, which follows, for example, by using Propositions 2.1 and 2.2 in \cite{AKN} for the endpoint cases $\alpha = \beta = 1$ and $\alpha = \beta = \infty$, along with real interpolation in the classical Sobolev spaces (see \cite{DeVoreScherer}).

From \eqref{e:weakRHS'}, \eqref{e:LHSsemi} and \eqref{e:RHSsemi}, and using the assumption that $\alpha = \alpha^*(p,q)$ along with the scaling condition $2s = d - (\frac{2}{p} + \frac{d}{q})$, we obtain \eqref{e:weakRHS_KT} in the limit $h \to 0$.
\end{proof}

\subsection{On $[O,A]$}
\begin{proposition} \label{p:OAfail}
Suppose $d \geq 3$. Then 
\begin{equation} \label{e:OAr>1fail}
\bigg\|\int_{\mathbb{R}^d}f(x-tv,v) \,\frac{\mathrm{d}v}{|v|^{2s}}\bigg\|_{L^{p,\infty}_tL^{q,\infty}_x}\lesssim \|f\|_{L^{p,r}}
\end{equation}
fails whenever $(\frac{1}{q},\frac{1}{p})$ belongs to $[O,A]$, $\frac{2}{p} + \frac{d}{q} = d - 2s$, and for all $r > 1$. Consequently, if $(\frac{1}{q},\frac{1}{p})$ belongs to the line $[O,A]$ and $2s = d - (\frac{2}{p} + \frac{d}{q})$, then the estimate
\[
\bigg\|\sum_j\lambda_j|e^{it\Delta}f_j|^2\bigg\|_{L^{p,\infty}_tL^{q,\infty}_x} \lesssim \|\lambda\|_{\ell^{p,r}}
\]
for any orthonormal system $(f_j)_j$ in $\dot{H}^s(\mathbb{R}^d)$ and any sequence $\lambda=(\lambda_j)_j$ in $\ell^{p,r}(\mathbb{C})$ fails for all $r > 1$.
\end{proposition}

\begin{proof}
Once we show that \eqref{e:OAr>1fail} fails, then from Proposition \ref{p:semiclassical} we may conclude that $\mathcal{O}^s((p,\infty),(q,\infty);(p,r))$ fails. To show \eqref{e:OAr>1fail} fails, we show that the dual estimate
\begin{equation*}
\bigg\|\frac{1}{|v|^{2s}}\int_{\mathbb{R}} g(x+tv,t) \, \mathrm{d}t\bigg\|_{L^{p',r'}}\lesssim \|g\|_{L_t^{p',1}L_x^{q',1}}
\end{equation*}
fails, for each $(\frac{1}{q},\frac{1}{p})$ belonging to $[O,A]$, $\frac{2}{p} + \frac{d}{q} = d - 2s$, and $r > 1$. For this, we consider $g(x,t)=e^{-\pi(t^2+|x|^2)}$. If we define
\[
G(x,v):=\frac{1}{|v|^{2s}}\int_{\mathbb{R}} e^{-\pi t^2} e^{-\pi|x+tv|^2} \, \mathrm{d}t,
\]
then a direct computation gives
\[
G(x,v) = \frac{e^{-\pi|x|^2}e^{\pi\frac{(x \cdot v)^2}{\langle v \rangle^2}}}{|v|^{2s}\langle v\rangle},
\]
where $\langle v \rangle := (1 + |v|^2)^{1/2}$. If we fix $v\in\mathbb{R}^d$, then by a rotation and change of variables it is easy to see
\begin{align*}
|\{x\in\mathbb{R}^d : |G(x,v)| \geq \lambda\} \sim \langle v\rangle \bigg(\log \frac{1}{\lambda|v|^{2s}\langle v\rangle}\bigg)_+^{\frac{d}{2}}.
\end{align*}
Thus, changing variables $v\mapsto \lambda^{-\frac{1}{2s+1}}v$, we obtain that 
\begin{align*}
|\{(x,v) \in \mathbb{R}^d \times \mathbb{R}^d : |G(x,v)|\geq \lambda\}| \sim \lambda^{-\frac{d+1}{2s+1}}\int_{\mathbb{R}^d}\langle v \rangle_\lambda \bigg(\log\frac{1}{ |v|^{2s}\langle v \rangle_\lambda}\bigg)_+^{\frac{d}{2}} \mathrm{d}v,
\end{align*}
where $\langle v \rangle_\lambda := (\lambda^{\frac{2}{2s+1}}+|v|^2)^{1/2}$. Since
$$
\lim_{\lambda \to 0+} \int_{\mathbb{R}^d}\langle v \rangle_\lambda \bigg(\log\frac{1}{ |v|^{2s}\langle v \rangle_\lambda}\bigg)_+^{\frac{d}{2}} \mathrm{d}v  \sim 1
$$
and, since on $OA$ we have $2s+1 = \frac{d+1}{p'}$, then it follows that 
\begin{align*}
\bigg\|\frac{1}{|v|^{2s}}\int_{\mathbb{R}} g(x+tv,t)\,\mathrm{d}t\bigg\|_{L^{p',r'}}^{r'} &\geq \int_0^\delta \left(\lambda \left|\left\{(x,v): |G(x,v)|\geq \lambda\right\}\right|^{\frac{1}{p'}}\right)^{r'}\frac{\mathrm{d}\lambda}{\lambda} \\
&\gtrsim \int_0^\delta \frac{\mathrm{d}\lambda}{\lambda} = \infty
\end{align*}
for sufficiently small $\delta>0$. On the other hand, it is easy to check that $\|g\|_{L_t^{p',1}L_x^{q',1}}$ is finite when $g(x,t) = e^{-\pi(t^2 + |x|^2)}$.
\end{proof}

\subsection{Failure of Conjecture \ref{conj:FLLS} when $d=1$}
Note that when $d=1$ the point $A$ is given by $(\frac{1}{q},\frac{1}{p}) = (0,\frac{1}{2})$; the next result shows that $\mathcal{O}^0(p,q;(p,1))$ fails at $A$ in this case.
\begin{theorem}\label{t:KTweakfail}
When $d=1$, the estimate
\begin{equation}\label{e:KTweakfail}
\bigg\|\int_{\mathbb{R}}f(x-tv,v)\,\mathrm{d}v\bigg\|_{L^2_t L^\infty_x}
\lesssim \|f\|_{L^{2,1}}
\end{equation}
fails, and hence Conjecture \ref{conj:FLLS} is false when $d=1$.
\end{theorem}

\begin{proof} By Proposition \ref{p:semiclassical}, once we show that \eqref{e:KTweakfail} fails, then the failure of Conjecture \ref{conj:FLLS} follows.

Let $\mathcal{E}_\delta$ be a $\delta$-neighbourhood of a Kakeya set $\mathcal{E} \subset \mathbb{R}^2$ with zero Lebesgue measure. This means
that, for any direction $\theta\in \mathbb{S}^1$, there exists a unit length line segment $\ell_\theta\subset \mathcal{E}$ such that $\theta$ and $\ell_\theta$ are parallel; the existence of such sets with zero Lebesgue measure goes back to \cite{Besicovitch}.

Assuming \eqref{e:KTweakfail}, and testing on $f=\chi_{\mathcal{E}_\delta}$, we see that
\begin{equation*}
\|\rho\chi_{\mathcal{E}_\delta}\|_{L^2_tL^\infty_x} \lesssim \|\chi_{\mathcal{E}_\delta}\|_{L^{2,1}_{x,v}} =
|\mathcal{E}_\delta|^\frac{1}{2},
\end{equation*}
where the notation $\rho$ was introduced in \eqref{e:rhodefn}. For the left-hand side, we claim that
\begin{equation*}
\sup_{x\in\R}\rho\chi_{\mathcal{E}_\delta}(x,t)\geq \frac{1}{(t^2+1)^\frac{1}{2}}
\end{equation*}
for all $t\in\mathbb{R}$, which quickly implies 
$
\|\rho\chi_{\mathcal{E}_\delta}\|_{L^2_t L^\infty_x} \gtrsim 1
$
uniformly in $\delta$, and hence, by taking a limit $\delta \to 0$ we obtain the desired contradiction.

To establish the remaining claim, we fix any $t\in\R$ and choose $\ell_t\subset \mathcal{E}_\delta$ so that $\ell_t$ and $(-t,1)$ are parallel. Further, we choose $x_t\in\R$ so that $x_t$ is an intersection point between the line extension of $\ell_t$ and the $x$-axis. Then it follows that 
\[
\sup_{x\in\R} \rho\chi_{\mathcal{E}_\delta}(x,t) \geq  \rho\chi_{\mathcal{E}_\delta}(x_t,t) \geq \int_{v_1}^{v_2} \chi_{\mathcal{E}_\delta}(x_t-vt,v)\,\mathrm{d}v = v_2-v_1,  
\] 
where we chose $v_1,v_2\in\R$ so that the line segment combining points 
$
(x_t-v_1t,v_1)
$
and 
$
(x_t-v_2t,v_2)
$ 
corresponds to the line segment $\ell_t$, and thus $v_2-v_1=(t^2+1)^{-\frac{1}{2}}$.
\end{proof}
The above argument shows that \eqref{e:KTweakfail} cannot be recovered even if we use the weak space $L^{2,\infty}_tL^\infty_x$ on the left-hand side.

\section{Weak-type estimates: Proof of Theorem \ref{t:smoothLorentz}} \label{section:AD}

We assume initially that $d \geq 3$ since we shall make use of the Keel--Tao endpoint $(\frac{1}{q},\frac{1}{p}) = (\frac{d-2}{d},1)$ for the classical estimates \eqref{e:classical}. Although this estimate fails when $d=2$, minor modification of the argument below gives the desired estimate.

By Proposition \ref{p:dualityprinciple}, it suffices to prove
\[
\| W_1 UU^* W_2\|_{\mathcal{C}^{p',1}} \lesssim \|W_1\|_{(2p',2), 2q'} \|W_2\|_{(2p',2), 2q'}
\]
whenever $\frac{2}{p} + \frac{d}{q} = d$ and $\frac{1}{q} \in (\frac{d-2}{d},\frac{d-1}{d+1})$. Here, $Uf(x,t) = e^{it\Delta}f(x)$ and therefore 
\[
UU^*F(x,t) = \int_\mathbb{R} e^{i(t-t')\Delta}F(\cdot,t')(x)\, \mathrm{d}t'.
\]
By relabelling, the desired estimate is equivalent to
\begin{equation}\label{e:dform}
\| W_1 UU^* W_2\|_{\mathcal{C}^{\sigma,1}} \lesssim \|W_1\|_{(u,2), r} \|W_2\|_{(u,2), r}
\end{equation} 
for $\sigma > d+1$, $u =2\sigma$, and $\frac dr+\frac 2u=1$. (Here, as in Section \ref{section:local}, we are using the notation $\|\cdot\|_{(p,r),q}=\|\cdot\|_{L^{p,r}_tL^q_x}$ and $\|\cdot\|_{p,q}=\|\cdot\|_{L^{p}_tL^q_x}$.) Estimates of this type were established by Frank--Sabin \cite{frank-sabin-2} by making use of interpolation along an analytic family of operators for which we only have a limited class of estimates. Here, we proceed more concretely and start by decomposing the operator $W_1 UU^*W_2$ by writing
\[
S_jF(x,t) = \int_\mathbb{R} \phi_j(t-t') e^{i(t-t')\Delta}F(\cdot,t')(x)\, \mathrm{d}t'
\]
where $\phi_j(t) = \chi_{[1,2]}(\frac{|t|}{2^j})$ and $j \in \mathbb{Z}$. This gives the decomposition $W_1 UU^*W_2 = \sum_{j \in \mathbb{Z}} W_1S_jW_2$ and allows us to have a wide enough class of estimates to apply bilinear real interpolation to obtain \eqref{e:dform}. In fact, we shall prove the somewhat stronger estimate
\begin{equation} \label{e:dformv2}
\sum_{j \in \mathbb{Z}} \|W_1S_jW_2\|_{\mathcal{C}^{\sigma,1}} \lesssim \|W_1\|_{(u,2), r} \|W_2\|_{(u,2), r}
\end{equation}
for $\sigma > d+1$, $u =2\sigma$, and $\frac dr+\frac 2u=1$. 

To begin the proof of \eqref{e:dformv2}, we define the set $\mathfrak Q \subset [0,\frac12]^3$ to be the convex hull of 
vertices 
\[ 
O=(0,0,0), \,\, Q_0=(0,\tfrac1d, 0), \,\, Q_0'=(\tfrac1d,0,0), \,\, Q_1=(\tfrac12,\tfrac12, 0), \,\, Q_2=(\tfrac12, \tfrac12, \tfrac12).
\] 
Also, we set $Q_3=(\frac1d, \frac1d, 0)$ and $Q_4=(\frac1{d+1},\frac1{d+1},\frac1{d+1})$. Note that the open line segment $(Q_3,Q_4)$ is contained in the interior of $\mathfrak Q$ and the estimate \eqref{e:dform} corresponds to the case that $(\frac1r, \frac1r,\frac1\sigma)\in [Q_3,Q_4)$. The following key lemma gives various estimates for the localised operator $S_j$. 
\begin{lemma}\label{l:basic} 
Let $d \geq 3$. If $(\frac1r, \frac1s,\frac1\sigma)\in \mathfrak Q$ then
\begin{equation}\label{e:basic-est} 
\| W_1 S_j W_2\|_{\mathcal{C}^\sigma} \lesssim 2^{(1 -\frac d2(\frac1r+\frac1s) - \frac 1u-\frac 1v)j} 
\|W_1\|_{u,r} \|W_1\|_{v,s} 
\end{equation}
holds whenever $u, v\ge \sigma$ and $\frac1u+\frac1v = \frac1\sigma$. 
\end{lemma}
\begin{proof}
The claimed estimates are consequences of interpolation between the estimates where $\sigma=\infty$ (at $O, Q_0, Q_0'$ and $Q_1$) and $\sigma=2$ (at $Q_2$). 

Suppose first $\sigma=\infty$, in which case $u = v =\infty$. At $O$, using the unitary property of the Schr\"odinger propagator and the Cauchy--Schwarz inequality, one can show that
\[
\|S_jF\|_{2,2} \lesssim 2^j \|F\|_{2,2}
\]
and therefore
\[
\|W_1S_jW_2\|_{\mathcal{C}^\infty} \lesssim 2^j\|W_1\|_{\infty,\infty} \|W_2\|_{\infty,\infty}.
\]
This gives \eqref{e:basic-est} at $O$.

Next, we consider $Q_1$. Using H\"{o}lder's inequality and the dispersive estimate for the Schr\"odinger propagator, 
\begin{align*}
|\langle F, W_1 S_j W_2 G \rangle_{L^2_{x,t}}| 
&\lesssim 2^{-\frac d2 j}
\int_\mathbb{R} \| F(t)W_1(t)\|_{L^1_x} \int_\mathbb{R} \|\phi_j(t-t')G(t')W_2(t')\|_{L^1_x} \, \mathrm{d}t'\mathrm{d}t.
\end{align*}
Therefore
\begin{align*}
|\langle F, W_1 S_j W_2 G \rangle_{L^2_{x,t}}| &\lesssim 2^{(1-\frac d2)j}\| FW_1\|_{L^2_tL^1_x} \int_{\mathbb{R}} \phi_0(\tau)\|GW_2(t-2^j\tau)\|_{L^2_tL^1_x}\, \mathrm{d}\tau \\
&\lesssim
2^{(1-\frac d2)j} \| F\|_{2,2}\|W_1\|_{\infty,2} \|G\|_{2,2} \|W_2\|_{\infty,2}
\end{align*}
and \eqref{e:basic-est} at $Q_1$ follows.

The remaining cases $Q_0$ and $Q_0'$ follow from the classical endpoint Strichartz estimate \eqref{e:classical} of Keel--Tao \cite{KeelTao}, which in its dual form states that
\begin{equation} \label{e:KTendpointdual}
\bigg\| \int_\mathbb{R} e^{-it'\Delta}F(\cdot,t')(x) \, \mathrm{d}t' \bigg\|_{L^2_x} \lesssim \|F\|_{L^2_tL^{\frac{2d}{d+2}}_x}.
\end{equation}

We give the details at $Q_0$, where $(\frac{1}{r},\frac{1}{s}) = (\frac{1}{d},0)$; the case $Q_0'$ follows by similar considerations. By writing
\[ 
\langle F, W_1 S_j W_2 G \rangle_{L^2_{x,t}} = 
\int_\mathbb{R} \bigg\langle e^{-it\Delta}F(t)W_1(t), \int_\mathbb{R} e^{-it'\Delta} \phi_j(t-t')G(t')W_2(t') \,\mathrm{d}t' \bigg\rangle\, \mathrm{d}t
\] 
and using the Cauchy-Schwarz inequality, Plancherel's theorem and \eqref{e:KTendpointdual}, we see
\begin{align*}
|\langle F, W_1 S_j W_2 G \rangle_{x,t}| &\lesssim \int_\mathbb{R} \| F(t)W_1(t)\|_{L^2_x} \bigg\|  \int_\mathbb{R} e^{-it'\Delta} \phi_j(t-t')G(t')W_2(t') \, \mathrm{d}t' \bigg\|_{L^2_x} \, \mathrm{d}t 
\\
&\lesssim \int \| F(t)W_1(t)\|_{L^2_x} \| \| \phi_j(t-t')G(t')W_2(t')\|_{L_{t'}^2L^{\frac{2d}{d+2}}_x} \, \mathrm{d}t 
\\
&\lesssim \| FW_1\|_{2,2} \| \phi_j(t-t')G(t')W_2(t')\|_{L^2_tL_{t'}^2L^{\frac{2d}{d+2}}_x} 
\\
&\lesssim 2^\frac j2 \|W_1\|_{\infty,\infty} \|W_2\|_{\infty,d} \| F\|_{2,2} \|G\|_{2,2}
\end{align*}
which gives the desired bound at $Q_0$.

Finally, we consider $\sigma=2$ and the point $Q_2$. By making use of the kernel of $UU^*$ we see that 
\begin{align*}
\| W_1 S_j W_2\|_{\mathcal{C}^2}^2 & \lesssim 2^{-dj} \int \phi_j(t-t')  |W_1(x,t)|^2 |W_2(y,t')|^2 \,\mathrm{d}x\mathrm{d}y\mathrm{d}t\mathrm{d}t' \\
& \lesssim 2^{(2-d-\frac2u-\frac2v)j} \|W_1\|_{u, 2}^2 \|W_2\|_{v, 2}^2
\end{align*}
for $2\le u,v \le \infty$ and $\frac1u+\frac1v\ge \frac12$. This completes the proof. 
\end{proof}

Finally, we use the estimates in Lemma \ref{l:basic} to obtain \eqref{e:dformv2} and hence Theorem \ref{t:smoothLorentz}. To this end, fix $(\frac1{r_\ast}, \frac1{r_\ast},\frac1{\sigma_\ast})\in (Q_3,Q_4)$ so that
\[ 
1=\frac d{r_\ast}+\frac 1{\sigma_\ast}, \qquad {r_\ast}\in (d, {d+1}) .
\] 
We need to show \eqref{e:dformv2} with $(u,r,\sigma)=(2\sigma_\ast, r_\ast, \sigma_\ast)$. 

Now choose $\sigma_0$ and $\sigma_1$ such that $\sigma_0 > \sigma_\ast> \sigma_1 \ge r_\ast$, and define
\[
\beta_0 :=1 -\frac d{r_*} -\frac1{\sigma_0}, \qquad \beta_1 :=1 -\frac d{r_*} -\frac1{\sigma_1}.
\]
Then $\beta_0 > 0 > \beta_1$. Also set $\kappa$ by 
\[ 
\frac1\kappa=\frac1{\sigma_1}-\frac1{2\sigma_0}.
\] 
From Lemma \ref{l:basic} we immediately obtain the estimates
\begin{align*}
2^{-\beta_0 j} \| W_1 S_j W_2\|_{\mathcal{C}^{\sigma_0} } & \lesssim 
\|W_1\|_{2\sigma_0,r_*} \|W_2\|_{2\sigma_0,r_*}, \\
2^{-\beta_1j} 
\| W_1 S_j W_2\|_{\mathcal{C}^{\sigma_1} }&\lesssim 
\|W_1\|_{2\sigma_0,r_*} \|W_2\|_{\kappa,r_*}, \\ 
2^{-\beta_1j} \| W_1 S_j W_2\|_{\mathcal{C}^{\sigma_1} } &\lesssim 
\|W_1\|_{\kappa,r_*} \|W_2\|_{2\sigma_0,r_*}, 
\end{align*}
which may be interpreted as the boundedness of the vector-valued bilinear operator $T : (W_1,W_2) \mapsto \{ W_1 S_j W_2\}_j$ as in \eqref{e:000}, \eqref{e:011} and \eqref{e:101}, with the spaces $A_j, B_j, C_j$ ($j=0,1$) given by
\[ 
A_0 = B_0= L^{2\alpha_0}_t L^{r_\ast}_x, \qquad A_1 = B_1= L^{\kappa}_t L^{r_\ast}_x 
\]
and
\begin{equation} \label{e:C0C1}
C_0 = \ell_{\beta_0}^\infty(\mathcal{C}^{\sigma_0}), \qquad C_1= \ell_{\beta_1}^\infty(\mathcal{C}^{\sigma_1}).
\end{equation}
Thus, \eqref{e:dformv2} follows once we show that $T$ is bounded
\begin{equation} \label{e:Tgoal}
L^{2\sigma_\ast,2}_tL^{r_\ast}_x \times L^{2\sigma_\ast,2}_tL^{r_\ast}_x \to \ell^1_0(\mathcal{C}^{\sigma_\ast,1}).
\end{equation}
In order to establish this, first notice that by choosing $\theta \in (0,1)$ such that 
\[ 
(1-\theta)\beta_0+\theta\beta_1=0, 
\]
or equivalently, $\frac{1-\theta}{\sigma_0} + \frac{\theta}{\sigma_1} = \frac{1}{\sigma_\ast}$, it follows that
\[
(C_0,C_1)_{\theta,1} = \ell^1_0((\mathcal{C}^{\sigma_0},\mathcal{C}^{\sigma_1})_{\theta,1}) = \ell^1_0(\mathcal{C}^{\sigma_\ast,1}).
\] 
For the first identity, see \cite{BerghLofstrom}, and for the second, we refer the reader to the work of Merucci \cite{Merucci}. Finally, we note that with the above choices of $\kappa$ and $\theta$, we have 
\[
(L^{2\sigma_0}_t L^{r_\ast}_x,L^{\kappa}_t L^{r_\ast}_x )_{\frac{\theta}{2},2} = L^{2\sigma_\ast,2}_tL^{r_\ast}_x.
\]
Here, note that the second exponent in mixed norm is fixed, so we may treat the mixed norm as a norm in vector-valued space. Hence, by Proposition \ref{p:bilinearint}, $T : (W_1,W_2) \mapsto \{ W_1 S_j W_2\}_j$ is bounded according to \eqref{e:Tgoal}; this implies \eqref{e:dformv2} and hence Theorem \ref{t:smoothLorentz} for $d \geq 3$.

Finally, we remark that when $d = 2$ the proof needs a very minor modification since \eqref{e:classical} fails when $(\frac{1}{q},\frac{1}{p}) = (0,1)$ and thus the proof of the claimed estimates at $Q_0$ and $Q_0'$ in Lemma \ref{l:basic} does not work as it stands. However, we may apply \eqref{e:classical} with $(\frac{1}{q},\frac{1}{p}) = (\varepsilon,1-\varepsilon)$ for all $\varepsilon \in (0,1)$ and inserting such estimates means we can obtain the analogous estimates at points $Q_{0,\varepsilon} = (\frac{1}{2} - \varepsilon,0)$ and $Q_{0,\varepsilon}' = (0,\frac{1}{2} - \varepsilon)$ for sufficiently small $\varepsilon > 0$. The same argument as used above then yields Theorem \ref{t:smoothLorentz} for $d=2$.

\begin{remark} The estimates in Theorem \ref{t:smoothLorentz} may be interpolated with other available estimates in the region $OCDA$ to give different types of weak-type bounds. It seems reasonable that progress could be made towards obtaining strong-type bounds, perhaps by exploiting further the estimates in Lemma \ref{l:basic} and different types of multilinear interpolation arguments. Alternatively, one may try to avoid the loss of information in passing from \eqref{e:dform} to \eqref{e:dformv2}; for example, by considering the $\mathcal{C}^2$ norm, it is natural to consider an $\ell^2$-sum on the left-hand side of \eqref{e:dformv2}. Proceeding in this way, we can somewhat refine the above argument to obtain
\[
\bigg\|\sum_j\lambda_j |e^{it\Delta}f_j|^2\bigg\|_{L^{p,p'}_tL^q_x} \lesssim \|\lambda\|_{\ell^{p}}
\]
whenever $(\frac{1}{q},\frac{1}{p}) \in [D,A)$.
\end{remark}

\section{Further results and remarks} \label{section:further}

\subsection{The case $p=\infty$}
When $p=\infty$, observe that $\alpha^*(\infty,q) = q$. We begin with an observation related to Lieb's generalised version of the Sobolev inequality in \eqref{e:Lieb_Sobolev}, showing that in the framework of estimates of the form $\mathcal{O}^s(\infty,q;(q,r))$, then the only possibility is that $r = 1$. Again, we first show the failure of the corresponding velocity average estimates.
\begin{proposition} \label{p:OBr>1fail}
Suppose $d \geq 1$. Then 
\begin{equation} \label{e:OBr>1fail}
\bigg\|\int_{\mathbb{R}^d}f(x-tv,v) \,\frac{\mathrm{d}v}{|v|^{\frac{d}{q'}}}\bigg\|_{L^\infty_tL^{q,\infty}_x}\lesssim \|f\|_{L^{q,r}}
\end{equation}
fails whenever $q \in (1,\infty]$ and $r > 1$. Consequently, if $q \in (1,\infty]$ and $2s = \frac{d}{q'}$, then the estimate
\[
\bigg\|\sum_j\lambda_j|e^{it\Delta}f_j|^2\bigg\|_{L^\infty_tL^{q,\infty}_x} \lesssim \|\lambda\|_{\ell^{q,r}}
\]
for any orthonormal system $(f_j)_j$ in $\dot{H}^s(\mathbb{R}^d)$ and any sequence $\lambda=(\lambda_j)_j$ in $\ell^{q,r}(\mathbb{C})$ fails for all $r > 1$.
\end{proposition}
\begin{proof}
By Proposition \ref{p:semiclassical}, it suffices to show that \eqref{e:OBr>1fail} fails, or equivalently, the failure of
\begin{equation} \label{e:OBdualLorentz}
\bigg\|\frac{1}{|v|^{\frac{d}{q'}}}\int_{\mathbb{R}} g(x+tv,t) \, \mathrm{d}t\bigg\|_{L^{q',r'}}\lesssim \|g\|_{L_t^{1}L_x^{q',1}}
\end{equation}
for $q,r \in (1,\infty]$. For this, we test on the functions $g_\varepsilon(x,t) = \frac{1}{\varepsilon} e^{-\pi\frac{t^2}{\varepsilon^2}} e^{-\pi|x|^2}$, for which an elementary change of variables gives $\|g_\varepsilon\|_{L_t^{1}L_x^{q',1}} \sim 1$ independent of $\varepsilon$. However, as $\varepsilon$ tends to zero, the left-hand side of \eqref{e:OBdualLorentz} converges to
\[
\bigg\|\frac{1}{|v|^{\frac{d}{q'}}}e^{-|x|^2}\bigg\|_{L^{q',r'}}
\]
which is infinite for $q \in (1,\infty]$.
\end{proof}
We complement this negative result with the following restricted weak-type result at $r = 1$.
\begin{proposition} \label{p:Liebalt}
Let $d \geq 1$. If $q \in (1,\infty)$ and $2s = d - \frac{d}{q}$, then 
\[
\bigg\|\sum_j\lambda_j||D|^{-s}e^{it\Delta}f_j|^2\bigg\|_{L^\infty_tL^{q,\infty}_x} \lesssim \|\lambda\|_{\ell^{q,1}}
\]
for any orthonormal system $(f_j)_j$ in $L^2(\mathbb{R}^d)$ and any sequence $\lambda=(\lambda_j)_j$ in $\ell^{q,1}(\mathbb{C})$.
\end{proposition}
Of course, Proposition \ref{p:Liebalt} implies that, whenever $q \in (1,\infty)$ and $2s = d - \frac{d}{q}$, then 
\[
\bigg\|\sum_j\lambda_j||D|^{-s}f_j|^2\bigg\|_{L^{q,\infty}_x} \lesssim \|\lambda\|_{\ell^{q,1}}
\]
holds for all orthonormal systems $(f_j)_j$ in $L^2(\mathbb{R}^d)$ and sequences $\lambda=(\lambda_j)_j$ in $\ell^{q,1}(\mathbb{C})$. This is closely related to Lieb's generalisation of the classical Sobolev inequality in \eqref{e:Lieb_Sobolev}. The possibility of upgrading the weak $L^{q,\infty}_x$ norm on the left-hand side to $L^q_x$ seems to be an interesting problem.
\begin{proof}[Proof of Proposition \ref{p:Liebalt}]
We fix $t \in \mathbb{R}$, $q_\ast \in (1,\infty)$ and let $2s = d - \frac{d}{q_\ast}$. Also, we let $q_0$ and $q_1$ be given by $\frac{1}{q_0} = \frac{1}{q_\ast} + \delta$ and $\frac{1}{q_1} = \frac{1}{q_\ast} - \delta$, where $\delta > 0$ is sufficiently small (so that $q_0, q_1 \in (1,\infty)$. By \eqref{e:localOAB} we obtain
\[
\bigg\|\sum_j\lambda_j|e^{it\Delta}P_k |D|^{-s}f_j|^2\bigg\|_{L^{q_i}_x} \lesssim 2^{(-1)^{i+1}\varepsilon_i k} \|\lambda\|_{\ell^{q_i}}
\] 
where $\varepsilon_i = (-1)^{i+1}(d-2s-\frac{d}{q_i}) = d\delta > 0$ for each $i=0,1$. From this, we deduce that
\[
\bigg\|\sum_j\lambda_j|e^{it\Delta}|D|^{-s}f_j|^2\bigg\|_{L^{q_\ast,\infty}_x} \lesssim \|\lambda\|_{\ell^{q_\ast,1}}
\] 
as desired. (For this final step, we are using an argument as in the proof of Proposition \ref{p:Bourgaintrick}, where we replace $L^p_tL^q_x$ with $L^q_x$.)
\end{proof}

\subsection{Weighted Strichartz estimates for the kinetic transport equation}

By putting together Theorem \ref{t:smooth} and Proposition \ref{p:semiclassical} we immediately obtain the following weighted estimates for the solution of the kinetic transport equation.
\begin{theorem}
Let $d \geq 1$. Suppose $(\frac{1}{p},\frac{1}{q})$ belongs to ${\rm int}\,OAB \cup [B,A)$. If $\alpha = \alpha^*(p,q)$ and $2s=d-(\frac{2}{p}+\frac{d}{q})$, then
\begin{equation*}
\bigg\|\int_{\mathbb{R}^d} f(x-tv,v)\, \frac{\mathrm{d}v}{|v|^{2s}} \bigg\|_{L^p_tL^q_x}\lesssim \|f\|_{L^{\alpha}}
\end{equation*}
for all $f\in L^{\alpha}$.
\end{theorem}

\subsection{The Schr\"odinger equation for the harmonic oscillator}
Corresponding to the Hermite operator $H = -\Delta + |x|^2$, we have the solution $e^{-itH}f$ of the Schr\"odinger equation for the quantum  harmonic oscillator
\[
\partial_t u + iHu = 0
\]
with initial data $u(x,0) = f(x)$, and where $x \in \mathbb{R}^d$, $d \geq 1$. The classical estimates in this case were proved by Koch--Tataru \cite{KochTataru} and may be stated as
\begin{equation} \label{e:Hermite}
\| |e^{-itH}f|^2 \|_{L^p_t((0,2\pi),L^q_x(\mathbb{R}^d))} \lesssim 1
\end{equation}
whenever $\|f\|_{L^2(\mathbb{R}^d)} = 1$, under the same conditions on $p$ and $q$ for \eqref{e:classical}; that is, $p,q \geq 1$ satisfy $\frac{2}{p} + \frac{d}{q} = d$ and $(p,q,d) \neq (1,\infty,2)$. Sj\"ogren and Torrea \cite{SjogrenTorrea} identified a transformation which facilitates a direct connection with the operator $e^{it\Delta}$, and thus the classical estimates \eqref{e:classical} and \eqref{e:Hermite} are equivalent. Using the same transformation, one may deduce an extension of \eqref{e:Hermite} to orthonormal systems of data. For brevity, we illustrate this by recording the following analogue of Theorem \ref{t:s=0}.
\begin{theorem}
Suppose $d\geq1$. If $p,q\geq1$ satisfy $\frac{2}{p}+\frac{d}{q}=d$, $1\leq q<\frac{d+1}{d-1}$ and $\alpha = \frac{2q}{q+1}$, then
\begin{equation*}
\bigg\|\sum_j\lambda_j|e^{-itH}f_j|^2\bigg\|_{L^p_t((0,2\pi),L^q_x(\mathbb{R}^d))}\lesssim\|\lambda\|_{\ell^\alpha}
\end{equation*}
holds for all orthonormal systems $(f_j)_j$ in $L^2(\mathbb{R}^d)$ and all sequences $ \lambda=(\lambda_j)_j$ in $\ell^\alpha(\mathbb{C})$. This is sharp in the sense that, for such $p,q$, the estimate fails for all $\alpha > \frac{2q}{q+1}$. Furthermore, when $q = \frac{d+1}{d-1}$, the estimate \eqref{e:FS} holds for all $\alpha < \frac{2q}{q+1}$ and fails when $\alpha = \frac{2q}{q+1}$.
\end{theorem}
If $K_{it}$ and $L_{it}$ denote the kernels of $e^{-itH}$ and $e^{it\Delta}$, respectively, then the proof rests on the transformation
\[
K_{i\sigma(t)}(x,x') = e^{-\frac{i}{2}t|x|^2}\langle t \rangle^{\frac{d}{2}} L_{i\frac{t}{2}}(\langle t \rangle x,x')
\]
for any $t > 0$. Here, $\sigma(t) = \frac{1}{2}\arctan (t)$ and $\langle t \rangle = (1 + t^2)^{\frac{1}{2}}$. From this, and using the scaling condition $\frac{2}{p}+\frac{d}{q}=d$, it follows (as in Theorem 1 in \cite{SjogrenTorrea} for a single function) that
\begin{equation*}
\bigg\|\sum_j\lambda_j|e^{-itH}f_j|^2\bigg\|_{L^p_t((0,\frac{\pi}{4}),L^q_x(\mathbb{R}^d))} = \bigg\|\sum_j\lambda_j|e^{it\Delta}f_j|^2\bigg\|_{L^p_t((0,\infty),L^q_x(\mathbb{R}^d))}.
\end{equation*}
To extend to $(0,2\pi)$ we use the fact that the kernel of $e^{-itH}$ satisfies $K_{-it}(x,x') = \overline{K_{it}(x,x')}$ and $K_{i(t + \frac{\pi}{2})}(x,x') = e^{i\frac{d\pi}{2}} K_{it}(-x,x')$, and the elementary fact that orthonormality of $(f_j)_j$ is preserved under complex conjugation.

\begin{acknowledgements}
This work was supported by JSPS Grant-in-Aid for Young Scientists A no. 16H05995 and JSPS Grant-in-Aid for Challenging Exploratory Research no. 16K13771-01 (Bez), grant from the Korean Government no. NRF-2017R1C1B1008215 (Hong), grants from the Korean Government no. NRF-2015R1\newline A2A2A05000956 and no. NRF-2015R1A4A1041675 (Lee), Grant-in-Aid for JSPS Research Fellow no. 17J01766 (Nakamura), and JSPS Grant-in-Aid for Scientific Research (C) no. 16K05209 (Sawano). Also, the first author would like to express his thanks to Fabricio Maci\`a for very enlightening discussions.
\end{acknowledgements}


\begin{thebibliography}{MMMMM}
\bibitem{AKN} L. Amour, M. Khodja, J. Nourrigat, \textit{The semi-classical limit of the time dependent Hartree--Fock equation: The Weyl symbol of the solution}, Anal. PDE, \textbf{6} (2013), no. 7, 1649--1674.

\bibitem{AJM} T. A\"issiou, D. Jakobson, F. Maci\`a, \textit{Uniform estimates for the solutions of the Schr\"odinger equation on the torus and regularity of semiclassical measures}, Math. Res. Lett. \textbf{19} (2012), 589--599.

\bibitem{BBGL_CPDE} J. Bennett, N. Bez, S. Guti\'errez, S. Lee, \textit{On the Strichartz estimates for the kinetic transport equation}, Comm. Partial Differential Equations \textbf{39} (2014), 1821--1826.

\bibitem{BerghLofstrom} J. Bergh, J. L\"ofstr\"om, \textit{Interpolation Spaces: An Introduction}, Springer--Verlag, New York, 1976.

\bibitem{Besicovitch} A. S. Besicovitch, \textit{On Kakeya's problem and a similar one}, Math. Z. \textbf{27} (1928), 312--320.

\bibitem{BLNS} N. Bez, S. Lee, S. Nakamura, Y. Sawano, \textit{Sharpness of the Brascamp--Lieb inequality in Lorentz space}, Electron. Res. Announc. Math. Sci. \textbf{24} (2017), 53--63.

\bibitem{Bourgain} J. Bourgain, \textit{Estimations de certaines functions maximales}, C. R. Acad. Sci. Paris \textbf{310} (1985) 499--502.

\bibitem{CastellaPerthame} F. Castella, B. Perthame, \textit{Estimations de Strichartz pour les \`equations de transport cin\'etique}, C. R. Acad. Sci. Paris S\'er. I Math. \textbf{332} (1996), 535--540.

\bibitem{CHP-1} T. Chen, Y. Hong, N. Pavlovi\'{c}, \textit{Global well-posedness of the NLS system for infinitely many fermions}, Arch. Ration. Mech. Anal. \textbf{224} (2017), 91--123.

\bibitem{CHP-2} T. Chen, Y. Hong, N. Pavlovi\'{c}, \textit{On the scattering problem for infinitely many fermions in dimension $d\geq3$ at positive temperature}, to appear in Ann. Inst. H. Poincar\'e Anal. Non Lin\'eaire. 


\bibitem{Cwikel74} M. Cwikel, \textit{On $(L^{p_0}(A_0),L^{p_1}(A_1))_{\theta,q}$}, Proc. Amer. Math. Soc. \textbf{44} (1974), 286--292. 


\bibitem{DeVoreScherer} R. DeVore, K. Scherer, \textit{Interpolation of linear operators on Sobolev spaces}, Ann. Math. \textbf{109} (1979), 583--599.


\bibitem{FLLS} R. Frank, M. Lewin, E. Lieb, R. Seiringer, \textit{Strichartz inequality for orthonormal functions}, J. Eur. Math. Soc. \textbf{16} (2014), 1507--1526.

\bibitem{frank-sabin-1} R. Frank, J. Sabin, \textit{Restriction theorems for orthonormal functions, Strichartz inequalities, and uniform Sobolev estimates}, to appear in Amer. J. of Math.

\bibitem{frank-sabin-2} R. Frank, J. Sabin, \textit{The Stein-Tomas inequality in trace ideals}, S\'eminaire Laurent Schwartz -- EPD et applications (2015-2016), Exp. No. XV, 12 pp., 2016.

\bibitem{GuoPeng} Z. Guo, L. Peng, \textit{Endpoint Strichartz estimate for the kinetic transport equation in one dimension}, C. R. Math. Acad. Sci. Paris \textbf{345} (2007), 253--256.

\bibitem{KeelTao} M. Keel, T. Tao \textit{Endpoint Strichartz estimates}, Amer. J. Math. \textbf{120} (1998), 955--980.

\bibitem{KochTataru} H. Koch, D. Tataru, \textit{$L^p$ eigenfunction bounds for the Hermite operator}, Duke Math. J. \textbf{128} (2005), 369--392.

\bibitem{Lee} S. Lee, \textit{Some sharp bounds for the cone multiplier of negative order in $\mathbb{R}^3$}, Bull. London Math. Soc. \textbf{35} (2003), 373--390.

\bibitem{LewinSabinWP} M. Lewin, J. Sabin, \textit{The Hartree equation for infinitely many particles. I. Well-posedness theory}, Comm. Math. Phys. \textbf{334} (2015), 117--170.

\bibitem{LewinSabinScatt} M. Lewin, J. Sabin, \textit{The Hartree equation for infinitely many particles. II. Dispersion and scattering in 2D}, Analysis \& PDE \textbf{7} (2014), 1339--1363.


\bibitem{Lieb_Sobolev} E. H. Lieb, \textit{An $L^p$ bound for the Riesz and Bessel potentials of orthonormal functions}, J. Funct. Anal. \textbf{51} (1983),  159--165.

\bibitem{LiebBAMS}  E. H. Lieb, \textit{The stability of matter: from atoms to stars}, Bull. Amer. Math. Soc. \textbf{22} (1990), 1--49.

\bibitem{Lieb-Thirring-1} E. H. Lieb, W. Thirring, \textit{Bound on kinetic energy of fermions which proves stability of matter}, Phys. Rev. Lett. \textbf{35} (1975), 687--689.

\bibitem{Lieb-Thirring-2} E. H. Lieb, W. Thirring, \textit{Inequalities for the moments of the eigenvalues of the Schr\"{o}dinger hamiltonian and their relation to Sobolev inequalities}, In: Studies in Mathematical Physics, Princeton Univ. Press (1976), 269--303.



\bibitem{Lions-Peetre} J. L. Lions, J. Peetre, \textit{Sur une classe d'espaces d'interpolation}, Inst. Hautes \'{E}tud. Sci. Publ. Math. \textbf{19} (1964), 5--68.

\bibitem{Merucci} C. Merucci, \textit{Interpolation dans $\mathcal{C}^\omega(H)$}, C. R. Acad. Sci. Paris S\'er. A-B \textbf{274} (1972), A1163--A1166.

\bibitem{MontSmith} S. J. Montgomery-Smith, \textit{Time decay for the bounded mean oscillation of solutions of the Schr\"odinger and wave equations}, Duke Math. J. \textbf{91} (1998), 393--408.

\bibitem{ONeil} R. O'Neil, \textit{Convolution operators and $L(p, q)$ spaces}, Duke Math. J. \textbf{30} (1963), 129--142.

\bibitem{Ovcharov_Endpoint} E. Ovcharov, \textit{Counterexamples to Strichartz estimates for the kinetic transport equation based on Besicovitch sets}, Nonlinear Anal. \textbf{74} (2011), 2515--2522.


\bibitem{sabin} J. Sabin, \textit{The Hartree equation for infinite quantum systems}, Journ\'{e}es \'{e}quations aux d\'{e}riv\'{e}es partielles, (2014), Exp. No. 8. 18p.

\bibitem{sabin-2} J. Sabin, \textit{Littlewood--Paley decomposition of operator densities and application to a new proof of the Lieb--Thirring inqeuality}, Math. Phys. Anal. Geom. \textbf{19} (2016), no. 2, Art. 11, 11 pp. 

\bibitem{simon} B. Simon, \textit{Trace ideals and their applications}, Vol. 35 of London Mathematical Society Lecture Note Series, Cambridge University Press, Cambridge, 1979.

\bibitem{SjogrenTorrea} P. Sj\"ogren, J. L. Torrea, \textit{On the boundary convergence of solutions to the Hermite--Schr\"odinger equation}, Colloq. Math. \textbf{118} (2010), 161--174.

\bibitem{SteinWeiss} E. M. Stein, G. Weiss, \textit{Introduction to Fourier Analysis on Euclidean Spaces}, Princeton Mathematical Series, No. 32, Princeton University Press, (1971).

\end{thebibliography}
\end{document}